\documentclass{amsart}[11pt]

\usepackage{amssymb,mathrsfs,amscd,graphicx,array, multirow,longtable,
enumerate, amsfonts, euscript}

 \usepackage{hyperref}
\makeatletter
\def\@tocline#1#2#3#4#5#6#7{\relax
  \ifnum #1>\c@tocdepth 
  \else
    \par \addpenalty\@secpenalty\addvspace{#2}%
    \begingroup \hyphenpenalty\@M
    \@ifempty{#4}{%
      \@tempdima\csname r@tocindent\number#1\endcsname\relax
    }{%
      \@tempdima#4\relax
    }%
    \parindent\z@ \leftskip#3\relax \advance\leftskip\@tempdima\relax
    \rightskip\@pnumwidth plus4em \parfillskip-\@pnumwidth
    #5\leavevmode\hskip-\@tempdima
      \ifcase #1
       \or\or \hskip 1em \or \hskip 2em \else \hskip 3em \fi%
      #6\nobreak\relax
    \dotfill\hbox to\@pnumwidth{\@tocpagenum{#7}}\par
    \nobreak
    \endgroup
  \fi}
\makeatother

\usepackage{letltxmacro}
\LetLtxMacro{\oldsqrt}{\sqrt}
\renewcommand{\sqrt}[2][]{\,\oldsqrt[#1]{#2}\,}

\newcommand{\abs}[1]{\lvert #1 \rvert}
\newcommand{\zmod}[1]{\mathbb{Z}/ #1 \mathbb{Z}}

\DeclareMathSymbol{\twoheadrightarrow} {\mathrel}{AMSa}{"10}

\DeclareMathOperator{\Pic}{Pic}

\DeclareMathOperator{\Aut}{Aut}
\DeclareMathOperator{\End}{End}
\DeclareMathOperator{\Hom}{Hom}

\DeclareMathOperator{\Gal}{Gal}
\DeclareMathOperator{\Mat}{Mat}

\DeclareMathOperator{\Tr}{Tr}

\DeclareMathOperator{\GL}{GL}

\def\a{{\mathfrak a}} 
\def\d{{\mathfrak d}} 

\newcommand{\m}{\mathfrak{m}}
\newcommand{\f}{\mathfrak{f}}

\newcommand{\BV}{\mathfrak{B}}

\newcommand{\BZ}{\mathscr{B}}
\newcommand{\CZ}{\mathscr{C}}

\newcommand{\EZ}{\mathscr{E}}

\newcommand{\RZ}{\mathscr{R}}

\newcommand{\cc}{\mathbb{C}}
\newcommand{\ff}{\mathbb{F}}

\newcommand{\nn}{\mathbb{N}}
\newcommand{\oo}{\mathbb{O}}

\newcommand{\qq}{\mathbb{Q}}
\newcommand{\rr}{\mathbb{R}}
\newcommand{\zz}{\mathbb{Z}}

\newcommand{\OO}{\mathcal{O}}
\newcommand{\PP}{\mathcal{P}}
\newcommand{\QQ}{\mathcal{Q}}
\newcommand{\RR}{\mathcal{R}}

\newcommand{\calO}{\mathcal{O}}
\newcommand{\calE}{\mathcal{E}}
\newcommand{\calH}{\mathcal{H}}
\newcommand{\calP}{\mathcal{P}}
\newcommand{\calL}{\mathcal{L}}
\newcommand{\wh}{\widehat}
\DeclareMathOperator{\Cl}{Cl}
\DeclareMathOperator{\Isog}{Isog}
\DeclareMathOperator{\Mass}{Mass}
\DeclareMathOperator{\Nr}{Nr}
\DeclareMathOperator{\Ell}{Ell}
\newcommand{\wt}{\widetilde}

\newcommand{\dieu}{Dieudonn\'{e} }
\newcommand{\grp}{\mathfrak{p}}
\newcommand{\grn}{\mathfrak{n}}
\newcommand{\grm}{\mathfrak{m}}
\newcommand{\grN}{\mathfrak{N}}
\newcommand{\grf}{\mathfrak{f}}

\def\makeop#1{\expandafter\def\csname#1\endcsname
  {\mathop{\rm #1}\nolimits}\ignorespaces}
\makeop{Hom}   \makeop{End}   \makeop{Aut}   \makeop{Isom}  \makeop{Pic} 
\makeop{Gal}   \makeop{ord}   \makeop{Char}  \makeop{Div}   \makeop{Lie} 
\makeop{PGL}   \makeop{Corr}  \makeop{PSL}   \makeop{sgn}   \makeop{Spf}
\makeop{Spec}  \makeop{Tr}    \makeop{Nr}    \makeop{Fr}    \makeop{disc}
\makeop{Proj}  \makeop{supp}  \makeop{ker}   \makeop{im}    \makeop{dom}
\makeop{coker} \makeop{Stab}  \makeop{SO}    \makeop{SL}    \makeop{SL}
\makeop{Cl}    \makeop{cond}  \makeop{Br}    \makeop{inv}   \makeop{rank}
\makeop{id}    \makeop{Fil}   \makeop{Frac}  \makeop{GL}    \makeop{SU}
\makeop{Nrd}   \makeop{Sp}    \makeop{Tr}    \makeop{Trd}   \makeop{diag}
\makeop{Res}   \makeop{ind}   \makeop{depth} \makeop{Tr}    \makeop{st}
\makeop{Ad}    \makeop{Int}   \makeop{tr}    \makeop{Sym}   \makeop{can}
\makeop{length}\makeop{SO}    \makeop{torsion} \makeop{GSp} \makeop{Ker}
\makeop{Adm}   \makeop{Mat}
\def\makebb#1{\expandafter\def
  \csname bb#1\endcsname{{\mathbb{#1}}}\ignorespaces}
\def\makebf#1{\expandafter\def\csname bf#1\endcsname{{\bf
      #1}}\ignorespaces} 
\def\makegr#1{\expandafter\def
  \csname gr#1\endcsname{{\mathfrak{#1}}}\ignorespaces}
\def\makescr#1{\expandafter\def
  \csname scr#1\endcsname{{\EuScript{#1}}}\ignorespaces}
\def\makecal#1{\expandafter\def\csname cal#1\endcsname{{\mathcal
      #1}}\ignorespaces} 

\def\doLetters#1{#1A #1B #1C #1D #1E #1F #1G #1H #1I #1J #1K #1L #1M
                 #1N #1O #1P #1Q #1R #1S #1T #1U #1V #1W #1X #1Y #1Z}
\def\doletters#1{#1a #1b #1c #1d #1e #1f #1g #1h #1i #1j #1k #1l #1m
                 #1n #1o #1p #1q #1r #1s #1t #1u #1v #1w #1x #1y #1z}
\doLetters\makebb   \doLetters\makecal  \doLetters\makebf
\doLetters\makescr 
\doletters\makebf   \doLetters\makegr   \doletters\makegr

\def\wt{\widetilde}

\def\wh{\widehat}

\newcounter{thmcounter} 
\numberwithin{thmcounter}{section}
\newtheorem{thm}[thmcounter]{Theorem}
\newtheorem{lem}[thmcounter]{Lemma}
\newtheorem{cor}[thmcounter]{Corollary}
\newtheorem{prop}[thmcounter]{Proposition}
\theoremstyle{definition}
\newtheorem{defn}[thmcounter]{Definition}

\newtheorem{rem}[thmcounter]{Remark}

\newtheorem{sect}[thmcounter]{}

\numberwithin{equation}{section}

\newtheoremstyle{notitle}  
  {}
  {}
  {\itshape}
  {}
  {}
  {\ }
  {.5em}
  {}
\theoremstyle{notitle}

\title[Class number formula]{Supersingular 
  abelian surfaces and Eichler class number formula}
\author{Jiangwei Xue,Tse-Chung Yang and Chia-Fu Yu}

\address{(Xue) Collaborative Innovation Centre of Mathematics, 
School of Mathematics and Statistics, Wuhan University, Luojiashan,
Wuhan, Hubei, 430072, P.R. China.}
\email{xue\_j@whu.edu.cn}

\address{(Yang) Institute of Mathematics, Academia Sinica,
  Astronomy-Mathematics Building, 6F, No. 1, Sec. 4, Roosevelt Road,
  Taipei 10617, TAIWAN.} 
\email{tsechung@math.sinica.edu.tw}

\address{(Yu) Institute of Mathematics,
  Academia Sinica and NCTS (Taipei Office), Astronomy-Mathematics
  Building, No. 1, Sec. 4, Roosevelt Road, Taipei 10617, TAIWAN.}
\email{chiafu@math.sinica.edu.tw} \address{
  The Max-Planck-Institut f\"ur Mathematik \\
  Vivatsgasse 7, Bonn \\
  Germany 53111} \email{chiafu@mpim-bonn.mpg.de}

\begin{document}

\date{\today} 

\begin{abstract}
  Let $F$ be a totally real field with ring of integers $O_F$, and $D$
  be a totally definite quaternion algebra over $F$.  A well-known
  formula established by Eichler and then extended by K\"orner
  computes the class number of any $O_F$-order in $D$. In this paper
  we generalize the Eichler class number formula so that it works for
  arbitrary $\zz$-orders in $D$. Our motivation is to count the
  isomorphism classes of supersingular abelian surfaces in a simple
  isogeny class over a finite prime field $\ff_p$. We give explicit
  formulas for the number of these isomorphism classes
  for all primes $p$. 
\end{abstract}

\maketitle


\tableofcontents

\section{Introduction}
\label{sec:intro}

Throughout this paper $p$ denotes a prime number. Let $\mathbf{D}$ be
the quaternion ${\mathbb Q}$-algebra ramified exactly at $\{p,
\infty\}$. For any supersingular elliptic curve $X$ over
$\bar{\ff}_p$, its endomorphism algebra
$\End_{\bar{\ff}_p}^0(X):=\End_{\bar{\ff}_p}(X)\otimes_\zz \qq$ is
isomorphic to $\mathbf{D}$, and the endomorphism ring
$\End_{\bar{\ff}_p}(X)$ is always a maximal order in
$\mathbf{D}$. The classical theory of Deuring establishes a one-to-one
correspondence between isomorphism classes of supersingular elliptic
curves over $\bar{\ff}_p$ and ideal classes of a maximal order
$\OO_{\mathbf{D}}\subset \mathbf{D}$.
Moreover, there is  an explicit formula for the class number
$h(\OO_{\mathbf{D}})$ as follows
\begin{equation}
  \label{eq:47}
   h(\OO_{\mathbf{D}})=\frac{p-1}{12}+\frac{1}{3}\left
  (1-\left(\frac{-3}{p}\right )\right )+\frac{1}{4}\left
  (1-\left(\frac{-4}{p}\right )\right ),
\end{equation}
where $\left( \frac{\cdot}{p}\right ) $ denotes the Legendre
symbol. In (\ref{eq:47}), the main term $(p-1)/12$ is the mass for
supersingular elliptic curves, which is also equal to
$\zeta_{\qq}(-1)(1-p)$, where $\zeta_{\qq}(s)$ is the Riemann zeta
function. The remaining terms are the adjustments for the isomorphism
classes with extra automorphisms. As the points corresponding to these
classes on the moduli space come from the reduction of elliptic fixed
points (whose $j$-invariants are $0$ or $1728$), the latter sum is
also called the elliptic part.

The goal of this paper is to provide an explicit description
and concrete formula for the isomorphism classes 
inside certain isogeny class of
supersingular abelian surfaces.  The main tools are the Honda-Tate theory
and extended methods in Eichler's class number formula.

Suppose that $q$ is a power of the prime number $p$. An algebraic
integer $\pi\in \bar{\qq}$ is said to be a \textit{$q$-Weil number} if
$\abs{\pi}=\sqrt{q}$ for all embeddings of $\qq(\pi)$ into $\cc$.  The
Honda-Tate theory \cite{honda1968,tate:ht} establishes a bijection
between isogeny classes of simple abelian varieties over $\ff_q$ and
conjugacy classes $q$-Weil numbers.  In \cite{waterhouse:thesis},
Waterhouse developed a theory for studying the isomorphism classes and
endomorphism rings of abelian varieties within a fixed simple isogeny
class. If $\pi$ is a $q$-Weil number, we denote by $X_\pi$ the abelian
variety over $\ff_q$ associated to $\pi$, unique up to isogeny. For
example, it is well known that every supersingular elliptic curve over
$\bar{\ff}_p$ admits a model over $\ff_{p^2}$ which lies inside the
isogeny class $\Isog(X_\pi)$ corresponding to the $p^2$-Weil number
$\pi=-p$. Then (\ref{eq:47}) may be interpreted as a formula for the
number of isomorphism classes in this isogeny class.  When $q=p$ is a
prime number, Waterhouse has proven the following result \cite[Theorem
6.1]{waterhouse:thesis}.
 
\begin{thm}\label{1.1}
  Suppose that $F=\qq(\pi)$ is not a totally real field. Then
\begin{enumerate}
  \item[(1)] The endomorphism algebra
    $\End^0_{\ff_p}(X_\pi)=\End_{\ff_p}(X_\pi)\otimes_\zz \qq$ of
    $X_\pi$ is commutative and coincides with $F$;
  \item[(2)] All orders in $F$ containing $R_0=\zz[\pi,p \pi^{-1}]$ are
    endomorphism rings;
  \item[(3)] There is a bijection between the set of $R_0$-ideal classes
    and the $\ff_p$-isomorphism classes of abelian varieties isogenous
    to $X_\pi$. 
\end{enumerate}
\end{thm}

In general there is no explicit description for $R_0$-ideal classes. 
However, the set of $R_0$-ideals is divided into finitely many genera and 
each genus has $h(R)$ ideal classes for 
some order $R$ containing $R_0$, 
where $h(R):=\abs{\Pic(R)}$ denotes the class number of the order
$R$. It is known that the class number $h(R)$ of $R$ is a multiple of 
the class number $h(F)$ of $F$. 
As a consequence of Waterhouse's result (Theorem~\ref{1.1}) the number of 
$\ff_p$-isomorphism classes in  $\Isog(X_\pi)$ 
is a multiple of the class number $h(F)$. 
Determining this multiple, nevertheless, requires 
an explicit description of genera of $R_0$-ideals. 
    
The above is the general picture when $F=\qq(\pi)$ is not totally real
for a $p$-Weil number $\pi$.  The exceptional case where $F$ is
totally real corresponds to the unique conjugacy class of the Weil
number $\pi=\sqrt{p}$, for which $F=\qq(\sqrt{p})$ is a real quadratic
field. It was already known to Tate \cite[Section 1,
Examples]{tate:ht} that $X_\pi$ in this case is a supersingular
abelian surface whose endomorphism algebra $\End_{\ff_p}^0(X_\pi)$ is
isomorphic to the quaternion algebra $D_{\infty_1,\infty_2}$ over $F$
ramified only at the two real places of $F$. Different from the
classical case of supersingular elliptic curves treated by Deuring,
Waterhouse \cite[Theorem 6.2]{waterhouse:thesis} showed that
$\End_{\ff_p}(X_\pi)$ is not always a maximal order in
$D_{\infty_1,\infty_2}$. A description of endomorphism rings of these
abelian surfaces will be given in Section~\ref{sec:isom_class}. Our
main result gives explicit formulas for the number of
$\ff_p$-isomorphism classes of this isogeny class.

\begin{thm}\label{1.2}
  Let $H(p)$ be the number of $\ff_p$-isomorphism classes of abelian
  varieties in the simple isogeny class corresponding to the $p$-Weil
  number $\pi=\sqrt{p}$. Then
(1)  $H(p)=1,2,3$ for $p=2,3, 5$, respectively;\\
(2)  For $p>5$ and $p\equiv 3 \pmod 4$, one has
    \begin{equation}
      \label{eq:1.1}
     H(p)=\frac{1}{2}h(F)\zeta_F(-1) +
     \left(\frac{3}{8}+\frac{5}{8}\left(2-\left(\frac{2}{p}\right)
    \right)\right)h(K_1)+\frac{1}{4}h(K_2)+\frac{1}{3}h(K_3),   
\end{equation}
where $K_j:=F(\sqrt{-j})$ for $j=1,2,3$, and $h(K_j)$ denotes the
class number of $K_j$.\\
(3) For $p>5$ and $p\equiv 1 \pmod 4$, one has
    \begin{equation}
      \label{eq:1.2}
   H(p)=       
      \begin{cases}
        8 \zeta_F(-1)h(F)+ h(K_1)+\frac{4}{3}
        h(K_3) & \text{for $p\equiv 1 \pmod 8$;} \\
      \left(\frac{45+\varpi}{2\varpi}\right) \zeta_F(-1)h(F)+\left
        (\frac{9+\varpi}{4\varpi}\right ) h(K_1)
      +\frac{4}{3} h(K_3) & \text{for $p\equiv 5 \pmod 8$;} \\
      \end{cases} 
   \end{equation}
   where $\varpi:=[O_F^\times:A^\times]$ and
   $A=\zz[\sqrt{p}]\subsetneq O_F$. The value of $\varpi$ is either
   $1$ or $3$. \\
The special value $\zeta_F(-1)$ of the Dedekind zeta-function
$\zeta_F(s)$ in both (2) and (3) can be calculated by 
Siegel's formula (\ref{eq:26}). 
\end{thm}

To obtain Theorem~\ref{1.2}, it is necessary to compute the class
number of $D_{\infty_1, \infty_2}$.
\begin{thm}\label{thm:class-num-for-quat-alg}
  Let $D=D_{\infty_1,\infty_2}$ be the quaternion algebra over
  $F=\qq(\sqrt{p})$ ramified only at the two real places of $F$. The
  class number $h(D)$ (i.e. the class number of any maximal order
  in $D$) is given below:
  \begin{enumerate}
  \item $h(D)=1, 2, 1$ for $p=2, 3, 5$, respectively; 
  \item if $p\equiv 1 \pmod{4}$ and $p\neq 5$, 
    $h(D)=h(F)\zeta_F(-1)/2 +h(K_1)/4+h(K_3)/3$;
  \item if $p\equiv 3 \pmod{4}$ and $p\neq 3$, then $h(D)=H(p)$ and is given by (\ref{eq:1.1}).
  \end{enumerate}
\end{thm}

\begin{rem}\label{1.4}
  It follows from the work of Herglotz \cite{MR1544516}
  that for all $p\geq 5$ and
  $j\in \{1,2,3\}$, we have $h(K_j)=\nu h(F)h(\Bbbk_j)$, where $\nu\in
  \{1,1/2\}$ and $\Bbbk_j:=\qq(\sqrt{-pj})$ 
  (see \cite[Section 2.10]{xue-yang-yu:num_inv}). 
  Hence one may factor
  out $h(F)$ in the results of Theorem~\ref{1.2} and
  \ref{thm:class-num-for-quat-alg}. For example, we get 
  \begin{equation}
    \label{eq:1.4}
    \frac{h(D)}{h(F)}=\frac{\zeta_F(-1)}{2} +
     \frac{h(\Bbbk_1)}{8}+\frac{h(\Bbbk_3)}{6}      
  \end{equation} 
for $p>5$ and $p\equiv 1 \pmod{4}$, and 
  \begin{equation}
    \label{eq:1.5}
     \frac{h(D)}{h(F)}=\frac{\zeta_F(-1)}{2} +
     \left(\frac{3}{8}+\frac{5}{8}\left(2-\left(\frac{2}{p}\right)
    \right)\right)h(\Bbbk_1)+\frac{h(\Bbbk_2)}{4}+\frac{h(\Bbbk_3)}{6}  
  \end{equation} 
for $p>5$ and $p\equiv 3 \pmod{4}$.
M. Peters pointed out that the formulas in the right hand
sides of (\ref{eq:1.4}) and (\ref{eq:1.5}) coincide with formulas for 
the proper class number $H^+(\d_F)$ of even definite quaternary 
quadratic forms of discriminant $\d_F$ 
(see \cite[p.~85 and p.~95]{chan-peters}), 
where $\d_F$ is the discriminant of
$F=\qq(\sqrt{p})$. That is, we have 
\begin{equation}
  \label{eq:1.6}
  h(D)=h(F)\, H^+(\d_F)\quad \text{for all primes $p>5$}.  
\end{equation}
Particularly, the number $h(D)/h(F)$ is always an integer.
The above formula for $H^+(\d_F)$ is obtained by Kitaoka 
\cite{kitaoka:nmj1973} for
primes $p\equiv 1 \pmod 4$ 
and by Ponomarev \cite{ponomarev:aa1976, ponomarev:aa1981} for all
primes $p$. 
Inspired by Peters' comment, we chased the literature and
discovered that formula (2) of Theorem~\ref{thm:class-num-for-quat-alg} 
was obtained in \cite{peters:aa1969}.  
%
\end{rem}

The calculations for both Theorem~\ref{1.2} and
\ref{thm:class-num-for-quat-alg} will be carried out in
Section~\ref{sec:explicit_formula}.  The main idea of the proof of
Theorem~\ref{1.2} is to apply Eichler's class number formula
(\cite{eichler:crelle55}, cf. \cite[Chapter V, Corollary 2.5,
p.~144]{vigneras}) for totally definite quaternion algebras.  Eichler
proved the class number formula for Eichler $O_F$-orders. Based on
Eichler's methods, K\"orner~\cite{korner:1987} worked out a similar
class number formula for any $O_F$-order.  However, the class number
formula established in \cite{korner:1987} is not readily applicable in
our case as the orders arising from the endomorphism rings of
supersingular abelian surfaces studied above do not necessarily
contain the ring of integers $O_F\subset F$. A main part of this
paper (Sections 2--5) is then devoted to proving a similar class
number formula and mass formula for arbitrary $\zz$-orders.  Our
generalized Eichler class number formula is the following.

\begin{thm}[Class number formula]\label{thm:class-number-formula-intro}
  Let $D$ be a totally definite quaternion algebra over a totally real
  number field $F$, and $\OO\subset D$ an arbitrary order in $D$ with
  center $A:=Z(\OO)$. The class number of $\OO$ is given by 
  \begin{equation}
    \label{eq:CNF}
h(\OO)=\Mass(\calO)+\frac{1}{2} \sum_{w(B)>1}(2-\delta(B))
  h(B)(1-w(B)^{-1})\prod_{p} m_p(B),    
  \end{equation}
  where the summation is over
  all the non-isomorphic orders $B$ whose fraction field $K$ is a
   quadratic extension of $F$  embeddable into
  $D$, and
  \[B\cap F=A, \qquad w(B):=[B^\times: A^\times]>1.\] Here
  $\Mass(\OO)$ is given by Definition~\ref{defn:mass} and can be
  computed by the mass formula (\ref{eq:rel_mass_formula}); $m_p(B)$
  is the number of conjugacy classes of local optimal embeddings
  (\ref{eq:48}); and $\delta(B)=1$ if 
  $B$ is closed under the complex conjugation $\iota\in \Gal(K/F)$,
  and $0$ otherwise.
\end{thm}



In the course of proving the class number formula we realize a subtle
point that the reduced norm of a $\zz$-order may strictly contain its
center. This causes some confusion as there are possibly more than one
choice for defining Brandt matrices and other terms as well at a few
places.  Thus one needs to examine all details in the original proof
in \cite{eichler:crelle55} (also \cite[Chapter V, Corollary 2.5,
p.~144]{vigneras}) until the final formula goes through.  Our
definition of Brandt matrices is justified by representation theory
(Section~\ref{sec:repr_int_brandt_matrix}).  We remark that the
methods of results here are algebraic, therefore all results in
Sections 2-4 make sense and remain valid when $F$ is replaced by an
arbitrary global function field, and $A$ by any $S$-order (whose
normalizer is the $S$-ring of integers), possibly except for
Theorems~\ref{thm:trac-brandt-matr-for1} and
\ref{thm:trac-brandt-matr-for2} and
Corollary~\ref{class_number_formula} in characteristic 2; also see
Remark~\ref{function_field}.
    
The second part (Section 6) of this paper is to
prove the explicit formulas given in Theorems~\ref{1.2} and
\ref{thm:class-num-for-quat-alg}.  
Based on our explicit formulas, we
used Magma to evaluate the numbers $H(p)$ for $p<10000$ and make the
tables for values of related terms for $p<200$.





\section{Preliminaries}
\label{sec:preliminaries}


\begin{sect} \label{sec:notations} \textbf{Notations and definitions.}
  Let $F$ be a number field with ring of integers $O_F$ and
  $A\subseteq O_F$ a $\zz$-order in $F$. Let $D$ be a
  finite-dimensional central simple $F$-algebra, and $\calO$ an
  $A$-order in $D$.  The order $\calO$ is said to be a {\it proper}
  $A$-order if $\calO\cap F=A$.  Similarly, for any finite field
  extension $K/F$, we say an order $B\subseteq O_K$ is a
  \textit{proper} $A$-order if $B\cap F=A$. An order $B$ is called a
  \textit{quadratic proper} $A$-order if $B$ is a proper $A$-order and
  the fraction field $K$ of $B$ is a quadratic extension of $F$. It
  does not necessarily mean that $B$ is an $A$-module generated by $2$
  elements. In fact, we will be interested only in those quadratic
  proper $A$-orders $B$ for which $K$ is a totally imaginary quadratic
  extension of $F$ in the case that $F$ is totally real.

  We will need the adelic language in the subsequent sections. For any
  place $v$ of $F$, denote by $F_v$ the completion of $F$ at $v$ and
  $O_v\subset F_v$ the ring of integers if $v$ is a finite place. Let
  $\wh{\zz}:=\varprojlim \zmod{n} =\prod_{p}\zz_p$ be the pro-finite
  completion of $\zz$.  Given any $\zz$-module $Y$, we
  write \[\wh{Y}:=Y\otimes_\zz \wh{\zz}=\prod_{p} Y_p, \qquad
  \text{where}\quad Y_p:=Y\otimes_\zz \zz_p. \] If $Y$ is also an
  $O_F$-module, then $Y_p$ further factors into $\prod_{v\mid p} Y_v$,
  where $Y_v:=Y\otimes_{O_F} O_v$. We are mostly concerned with the
  case where $Y$ is a finite-dimensional $\qq$-vector space or a
  $\zz$-module of finite rank. For example, $\wh{\OO}=\prod_p \OO_p$,
  $\wh{A}=\prod_p A_p$, and $\wh{\qq}=\wh{\zz}\otimes_\zz \qq$ is the
  ring of finite adeles of $\qq$. We also have that $\wh{F}=F\otimes_\zz
  \wh{\zz}= F\otimes_\qq \wh{\qq}=\prod_{v:\,{\rm finite}}' F_v$ is
  the ring of finite adeles of $F$, and $\wh{D}=D\otimes_\qq \wh{\qq}=
  D\otimes_F \wh{F}$ is the finite adele ring of $D$. Thus, $\wh
  \calO^\times \subset \wh D^\times$ and $\wh A^\times\subset \wh
  F^\times$ are open compact subgroups of the finite idele groups $\wh
  D^\times$ and $\wh F^\times$, respectively.

  A \textit{lattice} $I\subset D$ is a finitely generated $\zz$-module
  that spans $D$ over $\qq$. Its associated left order $\OO_l(I)$ is
  defined to be $\OO_l(I):=\{x\in D\mid xI\subseteq I\}. $ Similarly,
  one defines the associated right order $\OO_r(I)$. The lattice $I$
  is said to be a right $\OO$-ideal if $I\OO\subseteq I$.  A right
  $\OO$-ideal is not necessarily contained in $\OO$, and those that
  lie in $\OO$ are called \textit{integral} right $\OO$-ideals.

  Any right $\OO$-ideal $I$ is uniquely determined by its 
  completion $\wh{I}\subset \wh{D}$, as $I= \wh{I}\,\cap D$. For any $g\in
  \wh{D}^\times$, we set
  \[gI:= g\wh{I}\cap D, \qquad g\OO g^{-1}:=g\wh{\OO} g^{-1}\cap D.\]
  Then $gI$ is again a right $\OO$-ideal and $g\OO g^{-1}$ is an order
  in $D$.

  Given an ideal $\a\subsetneq A$, we write $A_\a$ for the $\a$-adic
  completion $\varprojlim A/\a^n$ of $A$, and $Y_\a:=Y\otimes_A A_\a$
  for any finitely generated $A$-module $Y$.

  If $S$ is a finite set, most of the time we write $\abs{S}$ for the
  cardinality of $S$, though sometimes it is more convenient
  to write it  as $\# S$.
\end{sect}




\begin{sect}\label{sec:locally-princ-latt} 
\textbf{Locally principal ideals.}
  A right $\calO$-ideal $I$ is said to be \textit{locally principal}
  with respect to $A=\OO\cap F$ if $I_{\grm}$ is a principal
  $\calO_\grm$-ideal for all maximal ideals $\grm$ of $A$.  Similarly,
  $I$ is said to be locally principal with respect to $\zz$ if $I_p$
  is a principal $\calO_p$-ideal for all primes $p$.  However, these
  two definitions are equivalent. Clearly one has the decomposition
  $\calO_p=\prod_{\grm|p}\calO_{\grm}$ arising from
  $A_p=\prod_{\grm|p} A_{\grm}$. It follows that the ideal $I_p$ is
  $\calO_p$-principal if and only if $I_{\grm}$ is
  $\calO_{\grm}$-principal for all $\grm|p$.  Thus, there is no
  confusion when $I$ is said to be a locally principal right $\calO$-ideal.

  Any locally principal right $\OO$-ideal $I$ is of the form $g\OO$
  for some $g\in \wh{D}^\times$. We have
  \begin{equation}
    \label{eq:29}
\OO_r(I)=\OO, \qquad \OO_l(I)= g\OO g^{-1}.     
  \end{equation}
Define $I^{-1}:=\OO g^{-1}$. Then $I^{-1}$ is a left $\OO$-ideal
whose associated right order is $\OO_l(I)$, and 
\begin{equation}
  \label{eq:30}
 I^{-1}I=\OO, \qquad I I^{-1}=g\OO g^{-1}= \OO_l(I).   
\end{equation}
Note that $I$ is a locally principal right $\OO_r(I)$-ideal if and
only if it is a locally principal left $\OO_l(I)$-ideal.  Thus if we
say (a lattice) $I$ is locally principal, without any reference to
orders, it is understood that $I$ is locally principal for both
$\OO_l(I)$ and $\OO_r(I)$.

Given two locally principal right $\OO$-ideals $I$ and $J$, we write
$I\simeq J$ if they are isomorphic as right $\OO$-ideals. This happens
if and only if there exist $g\in D^\times$ such that $gI=J$.  Denote
by $\Cl(\calO)$ the set of isomorphism classes of locally principal
right $\calO$-ideals in $D$.  The map $g \mapsto g\calO$ for $g\in
\wh{D}^\times$ induces a natural bijection
\[ D^\times \backslash \wh D^\times /\wh \calO^\times \simeq
\Cl(\calO). \] The class number of $\calO$ will be denoted by
$h=h(\calO):=\abs{\Cl(\calO)}$. 
\end{sect}


\begin{sect}\label{sec:norms-ideals}\textbf{Norms of ideals.}
  We study some properties of the norms of ideals in the present
  setting (the ground ring $A$ is not necessarily integrally closed).
  For any $A$-lattice $I$ in $D$, define the norm of $I$ (over $A$) by
  \[ \Nr_A(I):=\left \{\sum_{i=1}^m a_i \Nr(x_i)\ \text{for some $m\in
      \nn$} \, \Big | \, a_i\in A, \, x_i\in I \, \right \}\subset
  F, \] 
  where $\Nr:D\to F$ denotes the reduced norm map.  The
  formation of reduced norms of lattices commutes with completions. That
  is, for any ideal $\a\subsetneq A$, 
\begin{equation}
  \label{eq:completion}
  \Nr_A(I)_\a=\Nr_{A_\a}(I_\a). 
\end{equation}

The inclusion $\subseteq$ is obvious as $I\subseteq I_\a$. Since
$\Nr_A(I)$ is a finitely generated $A$-module, $\Nr_A(I)_\a=
\Nr_A(I)\otimes A_\a$ is the completion of $\Nr_A(I)$ with respect to
the $\a$-adic topology. In particular, $\Nr_A(I)_\a$ is closed in
$\Nr_{A_\a}(I_\a)$.  Let $\Nr_{\mathrm{Set}}$ be the set theoretic
image under the reduced norm map.  Note that $\Nr$ is continuous with respect
to the $\a$-adic topology, and $I$ is dense in $I_\a$.
 We have \[\Nr_{\mathrm{Set}}(I_\a)= \Nr_{\mathrm{Set}}(\bar{I}) \subseteq
\overline{\Nr_{\mathrm{Set}}(I)}\subseteq \Nr_A(I)_\a,\] where the
overline denotes the closure in the $\a$-adic topology.  Since
$\Nr_{A_\a}(I_\a)$ is spanned by $\Nr_{\mathrm{Set}}(I_\a)$ over
$A_\a$, we obtain the other
inclusion 
needed for the verification of (\ref{eq:completion}).


Let $\wt A_l:=\Nr_A(\calO_l(I))$ and $\wt
A_r:=\Nr_A(\calO_r(I))$. Clearly, $\Nr_A(I)$ is a module over the
ring $\wt A:=\wt A_l \wt A_r$. Here extra caution is needed since that
$\wt A_l$ (or $\wt A_r$) may strictly contain $A$ even if $\calO_l(I)$
(or $\OO_r(I)$) is a proper $A$-order.  An example will be given in
Section~\ref{sec:special_case} by taking $I=\mathbb{O}_8$, where
$\mathbb{O}_8$ is a certain nonmaximal order in the quaternion algebra
$D_{\infty_1, \infty_2}$. We do not know the relation between $\wt
A_l$ and $\wt A_r$ in general. However, $\Nr_A(I)$ is reasonably well
behaved when $I$ is locally principal.

Suppose that $I$ is a locally principal right ideal for a proper
$A$-order $\OO$. By (\ref{eq:29}), $\wt A=\wt A_l=\wt
A_r=\Nr_A(\OO)$. If we write $I=g\OO$ for some $g\in \wh{D}^\times$,
then $\Nr_A(I)= \Nr(g)\Nr_A(\OO)=\Nr(g)\wt A$. Hence $\Nr_A$ sends
locally principal right $\OO$-ideals to invertible $\wt A$-modules.
This property will enable us to define Brandt matrices for arbitrary
proper $A$-orders $\OO$ in Section~\ref{sec:trace_Brandt_mat}.
\end{sect}

\begin{sect}\label{sec:mult-prop}\textbf{Multiplicative properties.}
Let $I$ and $J$ be two $A$-lattices in $D$. We discuss 
when the multiplicative property $\Nr_A(I) \Nr_A(J)= \Nr_A(IJ)$ holds. 
Clearly $\Nr_A(I) \Nr_A(J)\subseteq \Nr_A(IJ)$ as $\Nr_A(I) \Nr_A(J)$
is generated by elements $\Nr(x)\Nr(y)=\Nr(xy)$ 
with $x\in I$, $y\in J$ and $xy\in
IJ$. 
Moreover, the equality can be checked locally:
the equality $\Nr_A(I) \Nr_A(J)= \Nr_A(IJ)$ holds if and only if
its local analogue $\Nr_{A_p}(I_p) \Nr_{A_p}(J_p)= \Nr_{A_p}(I_pJ_p)$
holds for every prime $p$. 
The product $IJ$ of $I$ and $J$ is said to be {\it coherent}
if
$\calO_r(I)=\calO_l(J)$ (cf. \cite[p.~183]{reiner:mo},
\cite[p.~22]{vigneras}).  We give an example which shows that
$\Nr_A(I)\Nr_A(J)\neq \Nr_A(IJ)$ when the product $IJ$ of $I$ and $J$
is not coherent, even though both $I$ and $J$ are locally principal
lattices.

Let $F=\qq$ and $D$ be any quaternion $\qq$-algebra with $D_p=\Mat_2(\qq_p)$. 
Take any two $\zz$-lattices $I$ and $J$ in $D$ with  
\[ I_p=
\begin{pmatrix}
  \zz_p & p\, \zz_p \\
  p^{-1} \zz_p & \zz_p \\
\end{pmatrix} \quad \text{and} \quad J_p=
\begin{pmatrix}
  \zz_p & \zz_p \\
  \zz_p & \zz_p \\
\end{pmatrix}. \] 
Then $\Nr_{\zz_p}(I_p)\Nr_{\zz_p}(J_p)=\zz_p$ but 
  $\Nr_{\zz_p}(I_p J_p)=p^{-1} \zz_p$ as 
\[ I_p J_p=\begin{pmatrix}
  \zz_p & \zz_p \\
  p^{-1}\zz_p & p^{-1} \zz_p \\
\end{pmatrix}. \] 
In this example the local product $I_p J_p$ is not
coherent and thus the global product $IJ$ is not coherent.
\end{sect}

Due to the above example we are content with the multiplicative
properties of the reduced norm for the type of products below.

\begin{lem}
  Suppose that the product $IJ$ of $I$ and $J$ is coherent and at
  least one of $I$ and $J$ is locally principal. Then
  $\Nr_A(IJ)=\Nr_A(I)\Nr_A(J)$.
\end{lem}
\begin{proof}
  Assume that $I$ is right locally $\calO$-principal, where
  $\calO=\calO_r(I)$.  For any prime $p$, one has \[\Nr_{A_p}(I_p
  J_p)=\Nr_{A_p}(x_p\calO_p J_p)=\Nr_{A_p} (x_p
  J_p)=\Nr(x_p)\Nr_{A_p}(J_p).\] Thus $\Nr_{A_p}(I_p
  J_p)=\Nr_{A_p}(I_p) \Nr_{A_p}(J_p)$ for all primes $p$ and hence
  $\Nr_A(IJ)=\Nr_A(I)\Nr_A(J)$. The case that $J$ is locally principal
can be proved similarly. 
\end{proof}
\begin{prop}[Criterion of units in $\OO$]
  We keep the notations of Section~\ref{sec:notations}, except that
  $F$ is allowed to be either a number field or a nonarchimedean local
  field. An element $u\in \OO$ is a unit if and only if $\Nr(u)\in
  O_F^\times$.
\end{prop}
\begin{proof}
  Let $S:=O_F[u]\subset D$ be the $O_F$-algebra generated by $u\in \OO$. Since
  $\OO$ is an order, $u$ is integral over $O_F$, and  $S$ is a
  finite $O_F$-algebra. Clearly, $u\in S^\times $ if and only if
  $\Nr(u)\in O_F^\times$. Let $R=S\cap \OO$, then $u\in R$ and $S$ is
  integral over the ring $R$. The proposition follows directly from
  Lemma~\ref{lem:unit-gp-integral-ext} below. 
\end{proof}

\begin{lem}\label{lem:unit-gp-integral-ext}
  Let $R\subseteq S$ be an inclusion of commutative rings with $S$
  integral over $R$. Then $R^\times =S^\times \cap R$.
\end{lem}
\begin{proof}
  Clearly, $R^\times \subseteq S^\times \cap R$. On the other hand,
  \[R^\times= R-\bigcup \m, \] where the union is  over all the
  maximal ideals $\m\subset R$.  Given an element $u\in S^\times \cap
  R$, to show that $u\in R^\times$, it is enough to show that
  $u\not\in \m$ for any maximal ideal $\m\subset R$.  Since $S$ is
  integral over $R$, by the going-up theorem \cite[Theorem
  5.10]{Atiyah-Mac}, any maximal ideal of $R$ can be obtained by
  intersecting a maximal ideal of $S$ with $R$.
\end{proof}

\section{Traces of Brandt matrices}
\label{sec:trace_Brandt_mat}
\numberwithin{thmcounter}{subsection}

 
In this section we define Brandt matrices for arbitrary orders in a
totally definite quaternion algebra and derive a formula for the trace
of Brandt matrices.  This allows us to obtain the generalized class
number formula as stated in
Theorem~\ref{thm:class-number-formula-intro}.  We follow closely
Eichler's original proof \cite{eichler:crelle55}; also see
Vign\'eras's book \cite{vigneras}.

\subsection{Brandt matrices}
\label{sec:12} 

Throughout the entire Section~\ref{sec:trace_Brandt_mat}, $F$ denotes a totally real
number field, $D$ a totally definite quaternion $F$-algebra,
$A\subseteq O_F$ a $\zz$-order in $F$ and $\calO$ a proper $A$-order
in $D$. Let $h=h(\calO)$ be the class number of $\calO$.

We fix a complete set of representatives $I_1,\dots, I_h$ for the right
ideal classes in $\Cl(\calO)$, and define
\begin{equation}
  \label{eq:34}
 \OO_i:=\OO_l(I_i),\qquad  w_i:=[\calO_i^\times: A^\times]. 
\end{equation}
The number $w_i$ only depends on the ideal class of $I_i$.  Since
$I_i=g_i\calO$ for some $g_i\in \wh D^\times$, we have $\OO_i=g_i\OO
g_i^{-1}$ by (\ref{eq:29}).  In particular, each $\calO_i$ is a proper
$A$-order, and if $\calO$ is closed under the canonical involution of
$D$, then each $\calO_i$ is also closed under the canonical
involution. Let
\begin{equation}\label{eq:33}
  \wt A:=\Nr_A(\OO)=\left \{\sum_{i=1}^m a_i \Nr(x_i)\ \text{for some
    $m\in \nn$}  
    \,\Big | \, x_i\in \calO, a_i\in A\, \right \}\subset F.
\end{equation} 
Then $\wt A$ is an order in 
$F$ with $A\subseteq \wt A\subseteq O_F$. For each
$i=1,\ldots, h$,  $\Nr_A(I_i)=\Nr(g_i)\wt{A}$ is an invertible
$\wt{A}$-module, and $\Nr_A(\OO_i)=\wt{A}$. 

\begin{lem}\label{lem:O-closed-under-can-inv}
  We have $\wt{A}=A$ if and only if $\OO$ is closed under the
  canonical involution $x\mapsto \Tr(x)-x$.
\end{lem}
\begin{proof}
Suppose that $\wt{A}=A$, then $\Nr(x)\in A$ for all $x\in
  \OO$. Therefore, 
\[\Tr(x)-x=\Nr(1+x)-\Nr(x)-1-x\in \OO.\]
On the other hand, suppose that $\OO$ is closed under the canonical
involution. Then for any $x\in \OO$, $\Nr(x)=(\Tr(x)-x)x$ lies in
$\OO$, and hence $\Nr(x)\in \OO\cap F=A$. It follows that
$\wt{A}=\Nr_A(\OO)=A$.
\end{proof}

In general, $\wt A$ is not necessarily equal to $A$. This is the
crucial difference in deriving the trace formula for Brandt matrices
over non-Dedekind ground rings. For brevity, we write $\Nr(I)$ for
$\Nr_A(I)$.



\begin{prop}\label{brandt_matrix}
Let $\grn$ be a locally principal integral $\wt A$-ideal. For any two integers  $i$ and $j$ with $1\le i, j\le h=h(\calO)$, 
there are bijections among the following finite sets:
 \begin{itemize}
\item [(a)] The set of locally principal right $\calO$-ideals
  $J\subseteq I_i$ such that $J\simeq I_j$ as right $\calO$-ideals and
  $\Nr(J)=\grn \cdot \Nr(I_i)$\,;
\item [(b)] The set of integral locally principal right $\calO_i$-ideals
  $J'\subseteq \calO_i$ such that $J'\simeq I_j I_i^{-1}$ as right
  $\calO_i$-ideals and $\Nr(J')=\grn$\,;
\item [(c)] The set of right principal $\calO_j$-ideals $J''\subseteq
  I_i I_j^{-1}$ such that $\Nr(J'')=\grn  \Nr(I_i)\cdot \Nr(I_j)^{-1}$\,;
\item [(d)] The set of right $\calO_j^\times$-orbits of elements
  $b\in I_i I_j^{-1}$ such that $\Nr(b\OO_j)=\grn \Nr(I_i)\Nr(I_j)^{-1}$.
\end{itemize}  
\end{prop}
\begin{proof}
  The bijection between (a) and (b) is given by $J\mapsto J':=J
  I_i^{-1}$. It is easy to see that the product $J I_i^{-1}$ is coherent
  and hence $\Nr(J I_i^{-1})=\Nr(J) \Nr(I_i)^{-1}$.
  The bijection between (a) and (c) is given by $J'':=J
  I_j^{-1}$. The bijection between (c) and (d) is given by $J''=b
  \calO_j$. 
\end{proof}

Perhaps it is helpful to indicate why the sets in the proposition
above are finite. This is already known if $A=O_F$ in K\"orner
\cite{korner:1987}.  Consider the set in (b). There are finitely many
ideals $J' O_F\subseteq \calO_iO_F$ with $\Nr(J'O_F)=\grn O_F$. As $c O_F
\subseteq A\subseteq O_F$ for some $c\in \nn_{>0}$, there are also finitely
many ideals $J'\subseteq \calO_i$ with $cJ'O_F\subseteq J'\subseteq J'O_F$
for each $J'O_F$.
     
\begin{defn}
Let $\BV_{ij}(\grn)$ be the cardinality of any of above finite sets. 
The \textit{Brandt matrix associated to $\grn$} 
is defined to be the matrix
\[\BV(\grn):=(\BV_{ij}(\grn))\in \Mat_h(\zz).   \]
\end{defn}
It follows from part (d) of Proposition~\ref{brandt_matrix} that 
\begin{equation}
  \label{eq:49}
\BV_{ii}(\grn)=\# \left(\{ b\in \calO_i | \Nr(b)\wt{A}=\grn
\}/\calO_i^\times\right).   
\end{equation}
In particular, $\BV_{ii}(\grn)\neq 0$ only if $\grn$ is principal and 
generated by a totally positive element.  

\subsection{Optimal embeddings}
\label{sec:opt_emb}
\def\Emb{{\rm Emb}}
Let $K$ be a quadratic CM extension of $F$ which can be embedded into
$D$ over $F$. Let $B$ be an $A$-order in $K$. 
Denote by $\Emb(B,\calO)$ the set of optimal embeddings from $B$
into $\calO$. In other words, 
\[\Emb(B,\OO):=\{\varphi\in \Hom_F(K, D)\mid \varphi(K)\cap \OO=\varphi(B)\}.\]
Equivalently, those are the embeddings of $A$-orders
$\varphi:B\hookrightarrow \calO$ so that $\calO/\varphi(B)$ has no
torsion. One can show that $\Emb(B,\calO)$ is a finite
set. Indeed, let $x\in B$ be a fixed element generating $K$ over $F$,
then each map $\varphi$ is uniquely determined by the image
$\varphi(x)$ in $\calO$. As the elements $\varphi(x)$, when $\varphi$
varies, have a fixed norm, these elements land in the intersection of
the discrete subset $\calO$ and a compact set in $\calO\otimes \rr$,
which is a finite set.  Note that $\Emb(B, \calO)$ is nonempty only
if $B$ is a proper $A$-order.  Moreover, if $\calO$ is closed under
the canonical involution, then $\Emb(B, \calO)$ is nonempty only if
$B$ is closed under the complex conjugation $\iota\in \Gal(K/F)$. 

The group $\calO^\times$ acts on $\Emb(B, \calO)$ from the right by
$\varphi\mapsto g^{-1}\varphi g$ for 
all $\varphi\in \Emb(B, \OO)$ and $g\in \OO^\times$.
We denote
\begin{gather*}
   m(B,\calO):=\abs{\Emb(B,\calO)}, \quad  \quad
m(B,\calO,\calO^\times):= \abs{
 \Emb(B,\calO)/\calO^\times},\\
w(B):=[B^\times: A^\times],\qquad\text{and} \qquad  w(\calO):=[\calO^\times:
A^\times].
\end{gather*}
 Then one has 
\begin{equation}
  \label{eq:4.1}   
  m(B, \calO, \calO^\times)= \frac{m(B,\calO)} {w(\calO)/w(B)}.
\end{equation}
Indeed, let $\varphi\in \Emb(B,\calO)$ be an element. The orbit
$O(\varphi)$ of $\varphi$ under the $\calO^\times$-action 
is isomorphic to $\calO^\times/\varphi(B)^\times$, and
hence $\abs{O(\varphi)}=[\calO^\times:
B^\times]=w(\calO)/w(B)$, which is independent of $\varphi$. 
This gives (\ref{eq:4.1}). As a result, one
obtains 
\begin{equation}
  \label{eq:4.2}
  \frac{m(B, \calO_i,
  \calO^\times_i)}{w(B)}=\frac{m(B,\calO_i)}{w_i}, \quad \forall\,
  i=1,\dots, h. 
\end{equation}
As $\calO_p$ and $\calO_{i, p}$ are isomorphic, one has $m(B_p,
\calO_p, \calO_p^\times)=m(B_p, \calO_{i,p}, \calO_{i,p}^\times)$ for
any $i=1,\dots, h$. For simplicity, we write 
\begin{equation}
  \label{eq:48}
m_p(B):=m(B_p, \calO_p,
\calO_p^\times).
\end{equation}

\begin{lem}\label{3.2.1} 
Let $h(B):=\abs{\Pic(B)}$ be the class number of $B$. We have
\begin{equation}
  \label{eq:4.3}
  \sum_{i=1}^h m(B,\calO_i, \calO_i^\times)=h(B)\prod_{p} m_p(B).
\end{equation}
\end{lem}

The proof is similar to that of \cite[Theorem 5.11, p.~92]{vigneras} 
(also see \cite{wei-yu:classno}, Lemma 3.2 and below), and is omitted. 

\subsection{Traces of Brandt matrices}
\label{sec:trace_brandt_mat}

Suppose that $\grn=\wt{A}\beta\subseteq \wt A$ is generated by a
totally positive element $\beta\in \wt{A}$. Choose a complete set
$S=\{\epsilon_1,\dots, \epsilon_s\}$ of representatives for the finite
group $\wt A^\times_+/(A^\times)^2$, where $\wt A^\times_+$ denotes
the subgroup of totally positive elements in $\wt A^\times$. We define
two sets:
\begin{align*}
  \CZ_i&:=\{b\in \calO_i | \Nr(b)=\varepsilon \beta\  \text{\ for some
  $\varepsilon\in S$} \},\\
\BZ_i&:=\{b \in \calO_i | \Nr(b)\wt{A}=\grn \}/A^\times. 
\end{align*}
Since $\ker(A^\times \xrightarrow{\Nr}
\wt{A}^\times)=\ker(A^\times\xrightarrow{a\mapsto a^2}A^\times)=\{\pm 1\}$, 
 \[\BZ_i\simeq\{b\in \calO_i | \Nr(b)=\varepsilon \beta\  \text{\
  for some 
  $\varepsilon\in S$} \}/\{\pm1\}=\CZ_i/\{\pm 1\},\]
and $\BV_{ii}(\grn)=\abs{\BZ_i}/w_i$ by (\ref{eq:49}). Thus, 
\begin{equation}
  \label{eq:27}
 \BV_{ii}(\grn)=\abs{\CZ_i}/2w_i.
\end{equation}

We define the symbol
\begin{equation}
  \label{eq:50}
 \delta_{\grn}=
\begin{cases}
  1 & \text{if $\grn=\wt{A}a^2$ for some $a\in A$}; \\
  0 & \text{otherwise.}
\end{cases}
\end{equation}
Note that the center of $\calO$ or $\calO_i$ is equal to $A$. It
follows that  
\begin{equation}
  \label{eq:4.4}
  2 \delta_\grn=\abs{\CZ_i \cap A}.
\end{equation}

Let $\calP_{\calO,\grn}$ be the set of characteristic polynomials of
non-central elements $b\in \CZ_i$ for some $i$. This is a finite set
in $\wt A[X]$ as for any $x\in \calO_i$, the reduced trace
$\Tr(x)=\Nr(x+1)-\Nr(x)-1\in \wt A$.  It is convenient to introduce a
slightly larger finite set which is independent of $\calO$ but depends
on $\grn$.  Let $\calP_{D,\grn}\subset \wt A[X]$ be the set consisting
of all irreducible polynomials of the form $X^2-t X+\varepsilon \beta$
for some $\varepsilon\in S$ such that $t^2-
4\varepsilon \beta\not \in F_v^2$
for all the ramified places $v$ of $F$ for $D$, including 
all the archimedean
ones. The set $\calP_{D,\grn}$ is again finite as the elements $t$ are
bounded for all the archimedean norms.  
Clearly $\calP_{\calO,\grn}\subseteq
\calP_{D,\grn}$.
      
For each $P\in \calP_{D,\grn}$, write $K_P:=F[X]/(P)$ and
$B_P:=A[x]\subset K_P$, where $x$ is the image of $X$ in $K_P$, and
thus a root of $P$ in $K_P$.  If $x'$ is the other root of
$P$ then $A[x']$ is isomorphic to $A[x]$ as $A$-orders. However, the
order $A[x']$ could be different from $A[x]$ in $K_P$.  For example,
let $p\equiv 5 \pmod{8}$, $F=\qq(\sqrt{p})$ with fundamental unit
$\epsilon\in O_F^\times$, and $A=\zz[\sqrt{p}]$ with $A^\times \neq
O_F^\times$. Then $A[\epsilon \zeta_6]\neq A[\epsilon \zeta_6^{-1}]$
as $A[\epsilon \zeta_6,\epsilon \zeta_6^{-1}]=O_F[\zeta_6]$ but both
orders are proper $A$-orders.
We would  like to
emphasize that $K_P$ is considered not just as an abstract field, but
rather a field with the distinguished element $x$. 

Local conditions imposed in the definition of $\calP_{D,\grn}$ ensure
the existence of an embedding of $(K_P)_v$ into $D_v$ locally
everywhere.  Then the local-global principle guarantees the existence
of an embedding of $K_P$ into $D$ as $F$-algebras. A priori, one
 needs to impose a further condition on $\calP_{D,\grn}$ so
that every order $B_P$ is a proper $A$-order.  However, omission of
this condition will not cause any trouble since $\Emb(B, \calO_i)$ is
empty if $B$ is not a proper $A$-order.
One  has the following equality for each $1\leq i \leq h$:
\begin{equation}
  \label{eq:4.5}
\coprod_{P\in \calP_{\calO,\grn}} \coprod_{B_P\subseteq B\subset K_P}
\Emb(B,\calO_i)
=\coprod_{P\in \calP_{D,\grn}} 
\coprod_{B_P\subseteq B\subset K_P} \Emb(B,\calO_i),
\end{equation}
as $\Emb(B,\calO_i)$ is nonempty only when $P\in \calP_{\calO,\grn}$.  

\begin{lem}\label{4.3}
  There is a natural bijection
  \begin{equation}
    \label{eq:4.6}
   \CZ_i-A \simeq \coprod_{P\in \calP_{D,\grn}} \coprod_{B_P\subseteq
  B\subset K_P} \Emb(B, \calO_i). 
  \end{equation}
\end{lem}
\begin{proof}
  To each element $b\in \CZ_i-A$, one associates a triple
  $(P,B,\varphi)$ in the right hand side as follows: $P$ is the
  characteristic polynomial of $b$,  $\varphi:K_P\to D$ is the
  $F$-embedding determined by $\varphi(x)=b$, where $x$ is the image of
  $X$ in $K_P$ and $B:= \varphi^{-1}(\calO_i)$, which ensures that
  $\varphi$ is an optimal embedding. 

 Conversely, to each
  triple $(P,B,\varphi)$ in the right hand side, one associates the
  element $b:=\varphi(x)$ in $\CZ_i-A$. Clearly, the element $b$ and
  the triple $(P, B,\varphi)$ determine each other uniquely and this
  gives a natural bijection between these two sets.
\end{proof}
\begin{defn}\label{defn:mass}
  The mass of $\OO$ is defined as
\[\Mass(\calO):=\sum_{i=1}^h\frac{1}{[\OO_i^\times :A^\times]}=\sum_{i=1}^h \frac{1}{w_i}.\]
\end{defn}

\begin{thm}[Eichler Trace Formula, first
  version]\label{thm:trac-brandt-matr-for1}
  We have $\Tr \BV(\grn)\neq 0$ only when the ideal $\grn$ is a
  principal and generated by a totally positive element. When
  $\grn$ is generated by a totally positive element $\beta$, the trace
  formula for $\BV(\grn)$ is given by 
\begin{equation}
  \label{eq:4.8}
   \Tr \BV(\grn)=\delta_\grn \cdot \Mass(\calO)+\frac{1}{2} \sum_{P\in
  \calP_{D,\grn}}\,\sum_{B_P\subseteq 
  B\subset K_P} M(B),    
\end{equation}  
where $\delta_\grn$ is defined by (\ref{eq:50}), and 
\begin{equation}\label{eq:23}
 M(B):=\frac{h(B)}{w(B)} \prod_{p} m_p(B).   
\end{equation}
\end{thm}
\begin{proof}
We have
\begin{equation}
  \label{eq:4.7}
  \begin{split}
  \BV_{ii}(\grn)&=\frac{\abs{
  \CZ_i}}{2w_i}=\frac{\abs{\CZ_i-A}}{2w_i}+\frac{2\delta_\grn}{2w_i} \\
  & =\frac{\delta_\grn}{w_i} + \frac{1}{2} \sum_{P\in
  \calP_{D,\grn}}\,\sum_{B_P\subseteq 
  B\subset K_P} \frac{\abs{\Emb(B, \calO_i)}}{w_i} \quad 
  \text{(\,Lemma~\ref{4.3}\,)} \\
  & =\frac{\delta_\grn}{w_i}+ 
  \frac{1}{2} \sum_{P\in \calP_{D,\grn}}\,\sum_{B_P\subseteq
  B\subset K_P} \frac{m(B, \calO_i, \calO_i^\times)}{w(B)} \quad
  (\text{by (\ref{eq:4.2})}). 
  \end{split}
\end{equation}
Summing over $i=1,\dots, h$ 
and by Lemma~\ref{3.2.1}, one obtains (\ref{eq:4.8}) for
the trace of the Brandt matrix $\BV(\grn)$.
\end{proof}

\begin{sect}
  We would like to count the right hand side of (\ref{eq:4.6}) by
  regrouping the elements according to the orders $B$.  For a fixed
  $1\leq i\leq h$, consider the quadruples $(B,P, \varphi, \alpha)$
  consisting of the following objects:
  \begin{enumerate}[(a)]
  \item a quadratic proper $A$-order $B$ with fraction field $K$,
    which is a totally imaginary quadratic extension of $F$ embeddable into $D$,
 \item a polynomial $P\in \PP_{D, \grn}$,  
 \item an optimal embedding $\varphi\in \Emb(B, \OO_i)$,
   \item an $F$-isomorphism $\alpha: K_P\to K$ such that $B_P\subseteq
    \alpha^{-1}(B) \subset K_P$.  Equivalently, $\alpha\in \Hom_A(B_P,
    B)$. 
  \end{enumerate}
  Clearly, each such quadruple defines a unique element $b\in \CZ_i-A$
  given by $b:= \varphi(\alpha(x))$.  Two quadruples $(B_r, P_r, \varphi_r,
 \alpha_r )_{r=1,2}$ are identified if $P_1=P_2$ and there
  exists an isomorphism $\rho: B_1\to B_2$ such that
  $\varphi_1=\varphi_2\circ \rho$, $\alpha_2=\rho\circ \alpha_1$. 
  
  Suppose that two quadruples $(B_r, P_r, \varphi_r,
  \alpha_r)_{r=1,2}$ give rise to the same $b\in \CZ_i-A$. Then
  necessarily $P_1=P_2$ since both are the characteristic
  polynomial of $b$.  Denote this polynomial by $P$. An $F$-embedding
  $K_P\hookrightarrow D$ is uniquely determined by the image of
  $x$. So $\varphi_1\circ \alpha_1=\varphi_2\circ \alpha_2$. In
  particular,
  \begin{equation}
    \label{eq:11}
  B_P\subseteq
  \alpha_1^{-1}(B_1)=\alpha_1^{-1}\varphi_1^{-1}(\OO_i)
  =\alpha_2^{-1}\varphi_2^{-1}(\OO_i)=
  \alpha_2^{-1}(B_2)\subset K_P.  
  \end{equation}
  So $B_1$ and $B_2$ are isomorphic.  Without lose of generality, we
  may assume that $B:=B_1=B_2$ from the very beginning. Note that
  $\varphi_1=\varphi_2$ implies that $\alpha_1=\alpha_2$ and vice
  versa.  Suppose that $\alpha_2=\iota\circ\alpha_1$, where $\iota\in
  \Gal(K/F)$ is the unique nontrivial isomorphism (i.e. the complex
  conjugation).  Then $\varphi_1=\varphi_2\circ\iota$, and it follows
  from (\ref{eq:11}) that $\iota(B)=B$.  On the other hand, if $(B,P,
  \varphi, \alpha)$ satisfies conditions (a)--(d) and $\iota(B)=B$,
  then $(B, P, \varphi\circ \iota,  \iota\circ \alpha)$ again satisfies
  these conditions, and the two quadruples give rise to the same
  element in $\CZ_i-A$.  




  
  Recall that $\grn=\wt{A}\beta$.  For each quadratic proper $A$-order
  $B$, let $T_{B,\grn}\subset B$ be the finite set
  \begin{equation}\label{eq:20}
    T_{B, \grn}:=\{x\in B-A\, |\,
  N_{K/F}(x)=\varepsilon\beta \ \ \text{for some $\varepsilon\in S$}\, \}, 
  \end{equation}
  and $\PP_{B, \grn}$ be the set of characteristic polynomials of
  elements in $T_{B,\grn}$. In general $\grn$ should be clear from the
  context, so we drop it from the subscript and write $T_B$ and
  $\PP_B$ instead.  We define
  \begin{gather}
    \CZ_{B,i}:=\{(P, \varphi, \alpha)\mid P\in \PP_B, \varphi\in \Emb(B, \OO_i),
    \alpha\in \Hom_A(B_P,B)  \}. 
  \end{gather}
  Note that if $P\in \PP_B$ but $P\not\in\PP_{\OO,
    \grn}$, 
  then $\Emb(B, \OO_i)=\emptyset $ for all $1\leq
  i\leq h$.  The fiber of the projection map $\CZ_{B,i}\to \PP_B$ over
  each $P\in \PP_B$ is
 \[
 \EZ_{B,P,i}:=\Emb(B, \OO_i)\times \Hom_A(B_P,B).\] The set
 $\EZ_{B,P,i}$ is equipped with an action of $\Gal(K/F)$ in the
 following way: if $\iota(B)=B$, then $\iota$ acts by sending
 $(\varphi, \alpha)\mapsto (\varphi\circ \iota, \iota\circ \alpha)$;
 otherwise $\iota$ acts trivially. It is clear that this action is
 independent of $P$ and $i$.  Let $\Gal(K/F)$ act on
 $\CZ_{B,i}$ fiber-wisely. We have
\begin{equation}\label{eq:21}
  \CZ_i-A\simeq \coprod_{B} \CZ_{B,i}/\Gal(K/F),
\end{equation}
where the disjoint union is taken over all the non-isomorphic
quadratic proper $A$-orders $B$.  In the next two subsections, we
calculate the cardinality of $\CZ_{B,i}/\Gal(K/F)$. There are two
cases to consider, depending on whether $\iota(B)=B$ or not.
\end{sect}
\begin{sect}
  Suppose that $\iota(B)=B$. We have 
\[\CZ_{B,i}/\Gal(K/F)=\coprod_{P\in \PP_B}\EZ_{B,P,i}/\Gal(K/F). \]
Note that $\Hom_A(B_P, B)=\Hom_F(K_P, K)$ for all $P\in \PP_B$ in this
case. Any choice of a fixed element $\alpha\in \Hom_F(K_P, K)$ induces
an isomorphism
  \begin{equation}
    \label{eq:16}
\Emb(B, \OO_i)\simeq \EZ_{B,P,i}/\Gal(K/F), \qquad \varphi\mapsto
(\varphi, \alpha).
  \end{equation}
  Therefore, \[\abs{\CZ_{B,i}/\Gal(K/F)}= \abs{\PP_B}\cdot \abs{\Emb(B, \OO_i)}.\]
Since $\iota(B)=B$, an element $b\in T_B$ if and only if $\iota(b)\in
T_B$. We have a surjective 2-to-1 map $T_B\to \PP_B$. It follows that 
\begin{equation}\label{eq:22}
  \abs{\CZ_{B,i}/\Gal(K/F)}= \frac{1}{2}\abs{T_B}\cdot \abs{\Emb(B, \OO_i)}.
\end{equation}
\end{sect}

\begin{sect}
  Suppose that $\iota(B)\neq B$.  Let $\QQ_B$ be the set of pairs
  $\{(P, \alpha)\mid P\in \PP_B, \alpha\in \Hom_A(B_P, B)\}$. Since
  $\Gal(K/F)$ acts trivially, we have 
\[\CZ_{B,i}/\Gal(K/F)=\CZ_{B,i}=\coprod_{(P, \alpha)\in \QQ_B}\Emb(B, \OO_i).\]
We claim that there is a canonical bijection between $T_B$ and
$\QQ_B$. Indeed, each pair $(P, \alpha)\in \QQ_B$ determines a unique
element $b:=\alpha(x)\in T_B$, where $x$ is the distinguished element
in $K_P$. On the other hand, given any element $b\in T_B$, we just set
$P$ to be the characteristic polynomial of $b$, and $\alpha: B_P\to B$
to be the canonical homomorphism sending $x$ to $b$. Therefore,  if
$\iota(B)\neq B$, then
\begin{equation}\label{eq:24}
  \abs{\CZ_{B,i}/\Gal(K/F)}= \abs{T_B}\cdot \abs{\Emb(B, \OO_i)}.
\end{equation}
\end{sect}






Let $\delta(B)$ be the
  symbol 
  \begin{equation}\label{eq:19}
\delta(B):=\begin{cases}
  1 \qquad  & \text{if $\iota(B)=B$};\\
0\qquad  & \text{otherwise}.
\end{cases}
  \end{equation}

  \begin{thm}[Eichler Trace Formula, second
    version]\label{thm:trac-brandt-matr-for2}
    Suppose that $\grn=\wt{A}\beta$ is generated by a totally positive
    element $\beta\in \wt{A}$. Let $\abs{T_{B,\grn}}$ be the
    cardinality of the set $T_{B,\grn}$ 
    defined in (\ref{eq:20}). The trace formula for $\BV(\grn)$ is
    given by
    \[\Tr\BV(\grn)=\delta_\grn \cdot \Mass(\calO)+\frac{1}{4}
    \sum_{B}(2-\delta(B)) M(B) \abs{T_{B,\grn}}. \]
Here in the last summation $B$ runs through all (non-isomorphic) 
quadratic proper $A$-orders which can be embedded into $D$. 
  \end{thm}
  \begin{proof}
    The proof employs the same line of arguments as
    Theorem~\ref{thm:trac-brandt-matr-for1}, except that instead of
    applying Lemma~\ref{4.3}, one combines (\ref{eq:21}),
    (\ref{eq:22}) and (\ref{eq:24}).
  \end{proof}
  Note that if $\calO$ is closed under the canonical involution of
  $D$, then $\wt{A}=A$ by Lemma~\ref{lem:O-closed-under-can-inv}. In
  this case, only those quadratic proper $A$-orders $B$ closed under
  the complex conjugation need to be considered in the trace formula,
  as $\Emb(B,\calO_i)$ is empty for all $1\leq i \leq h$ if
  $\delta(B)=0$. This observation applies to the class number formula below
  as well. 

  When $\grn=(1)=\wt{A}$, the Brandt matrix $\BV(\wt{A})$ is the identity
  and $\Tr \BV(\wt{A})=h(\calO)$. 

\begin{cor}[Class number formula] \label{class_number_formula} 
\begin{equation}
  \label{eq:4.9}
  \begin{split}
  h(\calO) & =\Mass(\calO)+\frac{1}{2} \sum_{P\in
  \calP_{D,(1)}}\,\,\sum_{B_P\subseteq B\subset K_P} M(B) \\
   & =\Mass(\calO)+\frac{1}{2} \sum_{ w(B)>1}(2-\delta(B))
  h(B)(1-w(B)^{-1})\prod_{p} m_p(B). \\      
  \end{split}
\end{equation}
Here in the last summation $B$ runs through all (non-isomorphic)
quadratic proper $A$-orders with
$w(B)=[B^\times:A^\times]>1$. Equivalently,
\begin{equation}
  \label{eq:18}
  h(\calO)=\Mass(\calO)+\frac{1}{2} \sum_{K}\sum_{\substack{B\subset K,\\ w(B)>1}}
  h(B)(1-w(B)^{-1})\prod_{p} m_p(B), \\      
\end{equation}
where $K$ runs through all (non-isomorphic) totally imaginary quadratic
extensions of $F$ embeddable into $D$, and $B$ runs through
all the distinct quadratic proper $A$-orders in $O_K$ with
$w(B)>1$. 
\end{cor}
\begin{proof}
  The first part of (\ref{eq:4.9}) follows directly from
  Theorem~\ref{thm:trac-brandt-matr-for1}. For each quadratic proper
  $A$-order $B$, let $q=w(B)$, and $B^\times/A^\times=\{\bar 1,\bar
  x_2, \dots, \bar x_q \}$. As the map $T_{B, (1)}\to \{\bar
  x_2,\dots, \bar x_q\}$ is surjective and two-to-one, sending $\pm x
  \mapsto\bar{x}$, one gets $\# T_{B, (1)}=2(q-1)$. So the second part
  of (\ref{eq:4.9}) follows from
  Theorem~\ref{thm:trac-brandt-matr-for2}. Formula (\ref{eq:18}) is
  just a more intuitive reformulation of (\ref{eq:4.9}).  Indeed, if
  $B\neq \iota(B)$, then both $B$ and $\iota(B)$ appears in the right
  hand side of (\ref{eq:18}), giving us $2 h(B)(1-w(B)^{-1})\prod_{p}
  m_p(B)$ for the isomorphic class of $B$.
\end{proof}

We call the sum in (\ref{eq:4.9}) the \textit{elliptic part} (of
the class number formula) and denote it by $\Ell(\OO)$. In other words,
\begin{equation}
\Ell(\OO):= \label{eq:28}
\frac{1}{2} \sum_{ w(B)>1}(2-\delta(B))
  h(B)(1-w(B)^{-1})\prod_{p} m_p(B). 
\end{equation}

\subsection{Local optimal embeddings}
\label{sec:15}

When $A=O_F$ and $\calO$ is an Eichler $O_F$-order of level $\grN$,
where $\grN\subseteq O_F$ is a square-free prime-to-$\scrD$ ideal, 
one has the formula  \cite[p.~94]{vigneras} for all prime ideals
$\grp\subset  O_F$, 
\[  m_\grp(B):=m(B_\grp, \calO_\grp, \calO_\grp^\times)=
\begin{cases}
  1-\left(\frac{B}{\grp}\right )  & \text{if
  $\grp|\scrD$;} \\
  1+\left(\frac{B}{\grp} \right ) &  \text{if
  $\grp|\grN$;}\\
  1 & \text{otherwise}.
\end{cases} \quad  \]
 Thus, one gets
\begin{equation}
  \label{eq:loc_opt}
  \prod_{\grp} m_\grp(B)=\prod_{\grp|\scrD} \left (
  1-\left(\frac{B}{\grp}\right )\right ) \prod_{\grp|\grN} \left (
  1+\left(\frac{B}{\grp}\right )\right ). 
\end{equation}
Here $(B/\grp)$ is the Eichler symbol, defined as follows:
\[ \left(\frac{B}{\grp}\right ):=
\begin{cases}
  \! \left(\frac{K}{\grp}\right ) & \text{if $\grp\nmid \grf(B)$;} \\
  \ 1 & \text{otherwise;}  
\end{cases} \]
where $\grf(B)\subseteq O_F$ is the conductor of $B$ and $(K/\grp)$ is
the Artin symbol 
\[ \left(\frac{K}{\grp}\right ):=
\begin{cases}
  \ 1 & \text{if $\grp$ splits in $K$;} \\
  -1 & \text{if $\grp$ is inert in $K$;} \\
  \ 0 & \text{if $\grp$ is ramified in $K$.} \\
\end{cases} \]

When $\calO$ is an Eichler $O_F$-order with arbitrary prime-to-$\scrD$
level $\grN$, 
Hijikata \cite[Theorem 2.3, p.~66]{hijikata:JMSJ1974} computed the numbers
of equivalence classes of  
the local optimal embeddings from $B_\grp$ into $\calO_\grp$. 

However, the situation is more delicate when $Z(\OO)=A\subsetneq O_F$. Let
$B\subset K$ and $\calO\subset D$ be proper $A$-orders.  Suppose that
$D_p\simeq \Mat_2(F_p)=\End_{F_p}(V_p)$, where $V_p$ is a free
$F_p$-module of rank two, and $\calO_p=\End_{A_p}(L_p)$, where
$L_p$ is a full $A_p$-lattice in $V_p$.

Fix an embedding $\varphi_0:K_p \to D_p$ of $F_p$-algebras. We view
$V_p$ as a free $K_p$-module of rank one through $\varphi_0$. A
lattice $M_p\subset V_p$ is said to be a \textit{proper $B_p$-lattice}
if $\{x\in K_p\mid \varphi_0(x)M_p\subseteq M_p\}=B_p$.  Let
$\calL(B_p,L_p,V_p)$ denote the set of isomorphism classes of proper
$B_p$-lattices $M_p\subset V_p$ such that there is an isomorphism
$M_p\simeq L_p$ of $A_p$-lattices. We claim that the number $m(B_p,
\calO_p, \calO_p^\times)$ is equal to $|\calL(B_p,L_p,V_p)|$.  Notice
that $m(B_p, \calO_p, \calO_p^\times)$ is the cardinality of
$\varphi_0(K_p)^\times \backslash \calE_p(B_p,
\calO_p)/\calO_p^\times$, where $\calE_p(B_p,\OO_p)=\{g\in D_p^\times
| \varphi_0(K_p)\cap g\OO_p g^{-1}=\varphi_0(B_p)\}$. 
It is straightforward to check that the map $g\mapsto
gL_p$ induces a bijection between the set $\varphi_0(K_p)^\times
\backslash \calE_p(B_p, \calO_p)/\calO_p^\times$ and
$\calL(B_p,L_p,V_p)$. This proves our claim.

We will need some structural theorems for modules over Bass orders. 
A standard reference for Bass orders is the original work
\cite{Bass-1962} of Bass. 
Recall that a $\zz$-order (or a $\zz_p$-order) 
$B$ is a \textit{Bass order} if $B$ is
Gorenstein and any order $B'$ containing $B$ is also Gorenstein. 
Bass orders share the following local property: $B$ is Bass if and
only if the completion $B_p$ is Bass for all primes $p$.
If a $\zz_p$-order $B_p$ is Bass,
then any proper $B_p$-module of rank one is isomorphic to $B_p$. 
Using this and our claim, we obtain the following lemma.

\begin{lem}\label{lem:local-optim-embedd-bass-order}
  Suppose that $\calO_p=\End_{A_p}(L_p)$.  If $B_p$ is a Bass order,
  then $m(B_p, \calO_p, \calO_p^\times)$ is either $0$ or $1$, and
  $m(B_p, \calO_p, \calO_p^\times)=1$ if and only if $B_p\simeq L_p$
  as $A_p$-modules.
\end{lem}



\section{Representation-theoretic interpretation of Brandt matrices}
\label{sec:repr_int_brandt_matrix}

\subsection{A general formulation}\label{trace}

Let $G$ be a unimodular locally compact topological group. Assume
there is a discrete and co-compact subgroup $\Gamma\subset G$. Then
the quotient $\Gamma\backslash G$ is a compact topological space with
right translation action by $G$.  Let $U\subset G$ be an open compact
subgroup.  Choose a Haar measure $dg$ on $G$ with volume one on $U$
and use the counting measure on $\Gamma$.  Since $\Gamma\backslash G$
is compact and $U$ is open, the double coset space $\Gamma\backslash
G/U$ is a finite set. Let $L^2(\Gamma\backslash G)$ be the Hilbert
space of square-integrable $\cc$-valued functions on the compact
topological space $\Gamma\backslash G$. The group $G$ acts on
$L^2(\Gamma\backslash G)$ by right translation, and we denote this
action by $R$.  The subspace $L^2(\Gamma\backslash G)^U$ of
$U$-invariant functions equals $L^2(\Gamma\backslash G/U)$, which is a
finite-dimensional vector space.  Let $\calH(G):=C^\infty_c(G)$ denote
the Hecke algebra of $G$, which consists of all smooth $\cc$-valued
functions on $G$ with compact support, together with the convolution.
The action of $\calH(G)$ on $L^2(\Gamma\backslash G)$ is as follows:
\[ ( R(f)\phi )(x)=\int_G f(g) \phi(xg) dg, \quad f\in \calH(G), \  
\phi\in L^2(\Gamma\backslash G). \]
Let $\calH(G,U)=C^\infty_c(U\backslash G/U)$ denote the subspace of
$U$-bi-invariant functions. For any $f\in \calH(G,U)$, the Hecke
operator $R(f)$ sends the finite-dimensional vector space 
$L^2(\Gamma\backslash G/U)$ into itself.

\subsection{Quaternion algebras, Brandt matrices and Hecke operators}
\label{sec:meaning_brandt_matrix}

Let $D$, $F$, $A$ and $\calO$ be as in Section~\ref{sec:12}. 
Note that $D^\times \subset \wh D^\times$ is not a discrete subgroup
when $[F:\qq]>1$ because the unit group $O_F^\times$ is not finite. 
We consider the following groups:
\[ G:=\wh D^\times/\wh A^\times, \quad \Gamma:=D^\times/A^\times,\quad
\text{and}\quad U:=\wh \calO^\times/\wh A^\times. \]
Then $\Gamma\subset G$ is a discrete and co-compact
subgroup. This allows us to consider Hecke operators on the space
$L^2(\Gamma\backslash G)$ of functions.    
The group $G$ operates transitively on the set of right locally
principal $\calO$-ideals. This gives natural bijections
\begin{equation}
  \label{eq:doublecoset}
  D^\times \backslash \wh D^\times /\wh \calO^\times \simeq
\Gamma\backslash G/U \simeq \Cl(\calO).
\end{equation}
Therefore, $h(\calO)=\dim L^2(\Gamma\backslash G/U)$. If ${\bf1}_U$
denotes the characteristic function of $U$, then the map $R({\bf1}_U)$ is the
identity on $L^2(\Gamma\backslash G/U)$ and $\Tr
R({\bf1}_U)=h(\calO)$. 

Let $\grn\subseteq \wt A$ be a locally principal integral $\wt A$-ideal.
The finite idele group $\wh{F}^\times$ operates on the set of $\wt
A$-ideals. Set
\[ U(\grn):=\{ x\in G\, |\, x \wh \calO \subseteq \wh \calO,\ 
\Nr(x) \wt A=\grn\, \}. \]
This is an open compact subset in $G$ which is stable under $U$ by
left and right action. Using the Cartan decomposition, one easily sees
that $U\backslash U(\grn)/U$ is a finite set. 
Let $g_1,\dots, g_h$ be a complete set of representatives for $D^\times
\backslash \wh D^\times /\wh \calO^\times$, one has
\[ \wh D^\times =\coprod_{i=1}^h D^\times g_i \wh \calO^\times,
\quad \text{and}\quad  G=\coprod_{i=1}^h \Gamma \bar g_i U, \]
where $\bar g_i$ are the images of $g_i$ in $G$. Set $I_i:=g_i
\calO$, then $I_1,\dots, I_h$ form a complete set of
representatives for ideal classes in $\Cl(\calO)$. 

Let $\chi_i$ be the characteristic function for the open compact subset 
$\Gamma\backslash \Gamma \bar g_i U\subset \Gamma\backslash G$. The
set 
$\{\chi_1,\dots, \chi_h\}$ forms a basis for the vector space 
$L^2(\Gamma\backslash G/U)$. Let $f$ be the characteristic function of $U(\grn)$,
which is an element in $\calH(G,U)$, and hence $R(f)$ 
is a linear operator on $L^2(\Gamma\backslash G/U)$. Write 
\[ R(f)\sim (a_{ij}) \]
for the representing matrix with respect to the basis
$\{\chi_i\}$. One has
\[ R(f)(\chi_j)=\sum_{i=1}^h a_{ij}\chi_i. \]
One computes
\[ R(f)(\chi_j)(x)=\int_G f(g) \chi_j(xg) dg=\int_{U(\grn)} \chi_j(xg)
dg. \]
Thus,
\[ a_{ij}=R(f)(\chi_j)(\bar g_i)=\int_{U(\grn)} \chi_j(\bar g_i g)
dg=\int_{U_{ij}} dg ={\rm vol}(U_{ij}), \] 
where 
\[ U_{ij}:=\{g\in U(\grn)\,|\, \bar g_i g\in \Gamma \bar g_j U\}. \]
Each $U_{ij}$ is invariant under right translation of $U$. 
For each fixed $i$ with $1\le i\le h$, the set
$U(\grn)$ is the disjoint union of $U_{ij}$ for $j=1,\dots, h$. For
$g\in U_{ij}$, one has 
\[ \bar{g}_ig\OO\simeq \bar{g}_j\OO=I_j, \quad \text{and}\quad  \bar g_i g
\calO \subseteq \bar g_i \calO=I_i. \]
If one puts $J:=\bar g_i g \calO$, then $\Nr(J)=\grn \Nr(I_i)$. 
As a
result we get a bijection
\[ U_{ij}/U \simeq \{J\subseteq I_i \, |\, J\simeq I_j,\  \Nr(J)=\grn
\Nr(I_i) \}, \quad \text{by} \ g\mapsto \bar g_i g \calO. \]   
Therefore, we get 
\[ a_{ij}=\abs{U_{ij}/U}=B_{ij}(\grn). \]

\begin{thm}\label{brandt_Hecke}
  Let $f$ be the characteristic function of $U(\grn)$ as above. Then the
  Brandt matrix is the representing matrix of the Hecke
  operator $R(f)$ with respect to the basis $\chi_1,\dots, \chi_h$ for
  the vector space $L^2(\Gamma\backslash G/U)$. 
\end{thm}

\begin{rem}\label{function_field}
  In the function field setting where 
  \begin{itemize}
  \item $F$ is a global function field
  with constant field $\ff_q$,
\item $A$ an $S$-order (whose normalizer
  is the $S$-ring of integers), where $S$ is a nonempty finite set of
  places of $F$,
\item $D$ a definite quaternion $F$-algebra relative to
  $S$, and 
\item $\calO$
  a proper $A$-order in $D$,
\end{itemize}
all results in Sections 3-4 make sense
  and remain valid, possibly except for
  Theorems~\ref{thm:trac-brandt-matr-for1} and 
  \ref{thm:trac-brandt-matr-for2} and
  Corollary~\ref{class_number_formula} in characteristic 2.  
\end{rem}




\section{Mass of Orders}
\label{sec:mass_orders}

\subsection{Mass formula}
\label{sec:mass_formula}

We keep the notations and assumptions of Section~\ref{sec:12}. In
particular, $\{I_1, \ldots, I_h\}$ is a complete set of
representatives for the right ideal classes in $\Cl(\OO)$, and
$\OO_i=\OO_l(I_i)$.   Recall that the \textit{mass} of
$\calO$ is defined by
\begin{equation}
  \label{eq:def_mass}
  \Mass(\calO)=\sum_{i=1}^h \frac{1}{w_i}, \qquad   w_i=[\calO_i^\times:
  A^\times]. 
\end{equation}
The mass of $\OO$ is
independent of the choices of representatives for $\Cl(\OO)$. 

\def\vol{{\rm vol}}  

\begin{lem}\label{vol_fund_domain}
  Let $G:=\wh D^\times/\wh A^\times$, $\Gamma:=D^\times/A^\times$ and
  $U:=\wh \calO^\times/\wh A^\times$. Then $\Gamma$ is a discrete
  cocompact subgroup of $G$, and for the
  counting measure on $\Gamma$ and any Haar measure on $G$, we have
  \begin{equation}
    \label{eq:vol_fund_domain}
    \vol(\Gamma\backslash G)=\vol(U)\cdot \Mass(\calO). 
  \end{equation}
\end{lem}
\begin{proof}
 By (\ref{eq:doublecoset}), 
one has $h=|\Gamma\backslash
G/U|$. Write $G=\coprod_{i=1}^h \Gamma g_i U$. Then 
\begin{equation}
  \label{eq:vol_comp}
  \begin{split}
  \vol(\Gamma\backslash G) =\sum_{i=1}^h \vol(\Gamma\backslash \Gamma
  g_i U) 
  =\sum_{i=1}^h \frac{\vol(U)}{\abs{\Gamma\cap g_i U g_i^{-1}}}.  
  \end{split}  
\end{equation}
The statement then follows from
$[\calO_i^\times:A^\times]=\abs{\Gamma\cap g_i U g_i^{-1}}$. 
\end{proof}
\begin{lem}\label{lem:mass-formulas-relative}
Let  $\RZ\subseteq \calO$ be two $\zz$-orders in $D$ with centers
$R$ and $A$, respectively. Then  
\begin{equation}
  \label{eq:rel_formula}
  \Mass(\RZ)=\Mass(\calO) 
   \frac{[\wh \calO^\times: \wh \RZ^\times]}{[A^\times:R^\times]}.  
\end{equation}
\end{lem}
\begin{proof}
  Let $G_1:=\wh{D}^\times/\wh A^\times$,
  $\Gamma_1:=D^\times/A^\times$, $U_1:=\wh \calO^\times/\wh A^\times$.
  We define $G_2$, $\Gamma_2$ and $U_2$ for the order $\RZ$ similarly.
  The map $G_2\to G_1$ is a finite cover with degree $[\wh A^\times:
  \wh R^\times]$ and $\Gamma_2\to \Gamma_1$ is a finite cover of
  degree $[A^\times:R^\times]$. Therefore, one gets
\[ \vol(\Gamma_2\backslash G_2)=\vol(\Gamma_1\backslash 
G_1)\frac{[\wh A^\times: \wh R^\times]}{[A^\times:R^\times]}. \]
On the other hand, $\vol(U_1)/\vol(U_2)=[\wh \calO^\times:  \wh
 \RZ^\times]/[\wh A^\times: \wh R^\times]$.    The lemma now
 follows from Lemma~\ref{vol_fund_domain}. 
\end{proof}

Let $\calO_{\max}$ be a maximal order in $D$ containing $\calO$. 
The mass formula \cite[Chapter V, Corollary 2.3]{vigneras} states that
\begin{equation}
  \label{eq:mass_formula}
  \Mass(\calO_{\max})=\frac{1}{2^{n-1}} |\zeta_F(-1)| h(F) \prod_{\grp
  | \scrD } 
  (N(\grp)-1),    
\end{equation}
where $\zeta_F(s)$ is the Dedekind zeta-function of $F$, $\scrD\subseteq
O_F$ is the discriminant ideal of $D$ over $F$ and $\grp$ ranges in
the set of prime ideals of $O_F$ that divide $\scrD$. Using
Lemma~\ref{lem:mass-formulas-relative}, one easily derives the
relative mass formula
\begin{equation}
  \label{eq:rel_mass_formula}
  \begin{split}
  \Mass(\calO) 
  & =\Mass(\calO_{\max})\cdot \frac{[\wh \calO_{\max}^\times: \wh
    \calO^\times]}{[O_F^\times:A^\times]} \\
  & =\frac{1}{2^{n-1}} |\zeta_F(-1)| h(F) \prod_{\grp | \scrD }
  (N(\grp)-1)\cdot \frac{[\wh \calO_{\max}^\times: \wh
    \calO^\times]}{[O_F^\times:A^\times]}.  
  \end{split}
\end{equation}

\subsection{Special cases}
\label{sec:special_case}
Let $F=\qq(\sqrt{p})$, where $p$ is a prime number,  and 
$D=D_{\infty_1,\infty_2}$, the totally definite quaternion $F$-algebra
ramified only at the archimedean places $\{\infty_1,\infty_2\}$. 
Let  $\mathbb{O}_1$ be a maximal $O_F$-order in $D$ and
$A=\zz[\sqrt{p}]\subseteq O_F$. 
By (\ref{eq:mass_formula}), the mass of $\mathbb{O}_1$ is
\begin{equation}
  \label{eq:25}
 \Mass(\mathbb{O}_1)=\frac{1}{2}\, \zeta_F(-1)\, h(F).   
\end{equation}

\begin{sect}\label{sec:5.2.1}
  \textbf{Mass of $\bbO_r$, $r=8,16$.}
Assume that $p\equiv 1 \pmod 4$ for the rest of this subsection. In this
case $A\neq O_F$, and $A/2O_F\cong \ff_2$. 
Let $\mathbb{O}_8, \mathbb{O}_{16}\subset \mathbb{O}_1$ be the proper
$A$-orders such that
\begin{gather}
  (\mathbb{O}_8)_2:=\mathbb{O}_8\otimes_{\zz} \zz_2=
\begin{pmatrix}
  A_2 & 2 O_{F_2} \\  O_{F_2} & O_{F_2}\\
\end{pmatrix},\qquad (\mathbb{O}_{16})_2=\Mat_2(A_2),\label{eq:31} \\  
(\mathbb{O}_r)_\ell=(\mathbb{O}_1)_\ell\qquad \forall\,\text{prime } \ell\neq 2, \quad r\in \{8,
16\}. \label{eq:32}
\end{gather}
The order $\mathbb{O}_r\subset \mathbb{O}_1$ is of 
index $r$.

We claim that $\Nr_A(\mathbb{O}_8)=O_F\neq A$. It is enough to show that
$\Nr_{A_\ell}((\mathbb{O}_{8})_\ell)=(O_F)_\ell$ for all primes $\ell$, which
follows from (\ref{eq:31}) for $\ell=2$, and  (\ref{eq:32}) for the
rest of the primes. 

Put $\varpi:=[O_F^\times: A^\times]$. 
By \cite[Section 4.2]{xue-yang-yu:num_inv}, we have $\varpi\in
\{1,3\}$, and $\varpi=1$ if $p\equiv 1 \pmod 8$.  By formula
(\ref{eq:rel_mass_formula}), one has
\begin{equation}
  \label{eq:5.10}
  \Mass(\mathbb{O}_r)=\Mass(\mathbb{O}_1)
\frac{[(\mathbb{O}_1/2\mathbb{O}_1)^\times:
  (\mathbb{O}_r/2\mathbb{O}_1)^\times]}{\varpi},\quad r=8,16.
\end{equation}
The group
$(\mathbb{O}_{16}/2\mathbb{O}_1)^\times\simeq \GL_2(\ff_2)$ and hence 
$|(\mathbb{O}_{16}/2\mathbb{O}_1)^\times|=6$. 

Suppose that  $p\equiv 1 \pmod 8$. 
The group 
$(\mathbb{O}_1/2\mathbb{O}_1)^\times\simeq \GL_2(\ff_2)\times
\GL_2(\ff_2)$ is of order $36$. 
By (\ref{eq:5.10}) we have
$\Mass(\mathbb{O}_{16})=6\, \Mass(\mathbb{O}_1)$. For the order
$\mathbb{O}_8$ one has  
\[ \mathbb{O}_8/2\mathbb{O}_1\simeq
\begin{pmatrix}
  \ff_2 & 0 \\ \ff_2\times \ff_2 & \ff_2\times \ff_2
\end{pmatrix}, \]
and hence $|(\mathbb{O}_8/2\mathbb{O}_1)^\times|=4$. Therefore by
(\ref{eq:5.10}) we
have $\Mass(\mathbb{O}_8)=9\, \Mass(\mathbb{O}_1)$.

Suppose now that $p\equiv 5 \pmod 8$. 
The group 
$(\mathbb{O}_1/2\mathbb{O}_1)^\times\simeq \GL_2(\ff_4)$ is of order $180$.
Thus,
$\Mass(\mathbb{O}_{16})=30/\varpi \cdot \Mass(\mathbb{O}_1)$. Since
\[ \mathbb{O}_8/2\mathbb{O}_1\simeq
\begin{pmatrix}
  \ff_2 & 0 \\ \ff_4 & \ff_4
\end{pmatrix},\]
we have  $|(\mathbb{O}_8/2\mathbb{O}_1)^\times|=12$. Thus,
$\Mass(\mathbb{O}_8)=15/\varpi \cdot \Mass(\mathbb{O}_1)$ by
(\ref{eq:5.10}). 

In summary, 
\begin{equation}
  \label{eq:mass_ord}
  \begin{split}
  & \Mass(\mathbb{O}_8)=
  \begin{cases}
    9/2\cdot \zeta_F(-1)\,h(F) & \text{for $p\equiv 1\pmod 8$};\\
  (15/2\varpi) \cdot \zeta_F(-1)\,h(F) & \text{for $p\equiv 5\pmod 8$};\\
  \end{cases} \\
  & \Mass(\mathbb{O}_{16})=
  \begin{cases}
    3\, \zeta_F(-1)\,h(F) & \text{for $p\equiv 1\pmod 8$};\\
  (15/\varpi) \cdot \zeta_F(-1)\,h(F) & \text{for $p\equiv 5\pmod 8$}.\\
  \end{cases}  
  \end{split}
\end{equation}
\end{sect}

\section{Supersingular abelian surfaces}
\label{sec:ss_ab_surf}

\subsection{Isomorphism classes}
\label{sec:isom_class}
Let $\pi=\sqrt{p}$ and $X_\pi$ an abelian variety over $\ff_p$
corresponding to the Weil number $\pi$. Let $\Isog(X_\pi)$ denote the
set of $\ff_p$-isomorphism classes of abelian varieties in the isogeny
class of $X_\pi$ over $\ff_p$. It is known that the endomorphism
algebra $D$ of $X_\pi$ over $\ff_p$ is isomorphic to the totally
definite quaternion algebra $D=D_{\infty_1.\infty_2}$ over
$F=\qq(\sqrt{p})$ defined in Section~\ref{sec:special_case}.  We also
recall the orders $\oo_1,\oo_8, \oo_{16}$ introduced there. The
endomorphism ring of each member $X$ in $\Isog(X_\pi)$ may be regarded
as an order in $D$, uniquely determined up to a inner automorphism of
$D$. Let $\mathfrak O_r$ denote the genus consisting of orders in $D$
which are locally isomorphic to $\oo_r$ at every prime $\ell$. If
$p\equiv 1\pmod{4}$, then $A=\zz[\sqrt{p}]\subset O_F$ is a
suborder of index $2$ in $O_F$.

We will need the following result, which is a special case of
\cite[Theorem 2.2]{yu:smf}. 
\begin{prop}\label{6.1}
  Let $X_0$ be an abelian variety over a finite field $\ff_q$ and
  $\RR:=\End_{\ff_q}(X_0)$ the endomorphism ring of $X_0$. Then there is
  a natural bijection from the set $\Cl(\RR)$ to the set of
  $\ff_q$-isomorphism classes of abelian varieties $X$ satisfying the
  following three conditions
  \begin{itemize}
  \item [(a)] $X$ is isogenous to $X_0$ over $\ff_q$,
  \item [(b)] the Tate module $T_\ell(X)$ is
  isomorphic to $T_\ell(X_0)$ as $\Gal(\bar{\ff}_q/\ff_q)$-modules
  for all primes $\ell\neq p$,
  \item [(c)] the Dieudonn\'e module $M(X)$ of $X$ is isomorphic to
    $M(X_0)$.   
  \end{itemize}
\end{prop}
\begin{thm}\label{explicit_description} \ 

  {\rm (a)} Suppose that $p\not \equiv 1 \pmod 4$. The endomorphism ring of
  any member $X$ in $\Isog(X_\pi)$ is a maximal order in $D$. Moreover, 
  there is a bijection between the set $\Isog(X_\pi)$ with the set 
  $\Cl(\oo_1)$ of ideal classes. 

  {\rm (b)} Suppose that $p\equiv 1 \pmod 4$. The endomorphism ring
  $\End(X)$ of
  any member $X$ in $\Isog(X_\pi)$ belongs to $\mathfrak O_r$ for some
  $r=1, 8, 16$. Moreover, for each $r\in \{1,8,16\}$  
  the set of members $X$ in $\Isog(X_\pi)$
  with $\End(X)\in \mathfrak O_r$ is in bijection with the set
  $\Cl(\oo_r)$ of ideal classes. In particular,   
  there is a bijection $\Isog(X_\pi)\simeq \coprod_{r=1,8,16} \Cl(\oo_r)$.
\end{thm}
\begin{proof}
  Part (a) has been proven in \cite[Theorem
  6.2]{waterhouse:thesis}. We prove part (b) where 
  $p\equiv 1 \pmod 4$. 
  By Proposition~\ref{6.1}, one is reduced to 
  classify the Tate modules and \dieu modules of members $X$ in
  $\Isog(X_\pi)$. Since the ground field is $\ff_p$, the \dieu module
  $M(X)$ of $X$ is simply an $A_p$-module in $F_p^2$. As $A_p$ is
  the maximal order in $F_p$, there is
  only one such isomorphism class and its endomorphism ring is a maximal
  order in $\Mat_2(F_p)$. The Tate module $T_\ell(X)$ of $X$ is simply an
  $A_\ell$-module. Therefore, when $\ell\neq 2$, there is only one such
  isomorphism class and its endomorphism ring is again a maximal order in
  $\Mat_2(F_\ell)$. Now we consider the case where $\ell=2$. Since 
  $2 O_{F_2}\subset A_2 \subset
  O_{F_2}$, the order $A_2$ is Bass and hence the classification of 
  $A_2$-modules is known; see \cite{Bass-1962}. It follows that  
  the Tate module $T_2(X)$ of $X$ is
  isomorphic to one of the following three $A_2$-lattices in $F_2^2$\,:
  \begin{equation}
    \label{eq:63}
 L_1=O_{F_2}^2, \quad L_2=A_2\oplus O_{F_2}, \quad L_4=A_2^2,     
  \end{equation}
(also see \cite[Corollary 5.2]{yu:sp_prime} for a direct classification).
One easily computes that $\End_{A_2}(L_1)=(\oo_1)_2$,  
$\End_{A_2}(L_2)=(\oo_8)_2$ and $\End_{A_2}(L_4)=(\oo_{16})_2$. 
If we let $X_1, X_8, X_{16}$ be members in $\Isog(X_\pi)$ representing
these three classes respectively and let $\RR_r:=\End(X_r)$, then
each $\RR_r\in \mathfrak O_r$ and the set of members $X$ in
$\Isog(X_\pi)$ defined as in
Proposition~\ref{6.1} is isomorphic to $\Cl(\RR_r)\simeq \Cl(\oo_r)$. This
proves part (b).  
\end{proof}


\subsection{Computation of class numbers}
\label{sec:explicit_formula}

In this subsection, we give explicit class number formulas for the
orders $\mathbb{O}_1$, $\mathbb{O}_8$ and $\mathbb{O}_{16}$ arising
from the study of supersingular abelian surfaces in the isogeny class
corresponding to $\pi=\sqrt{p}$. Recall that $\mathbb{O}_8$ and
$\mathbb{O}_{16}$ are necessary for consideration only when $p\equiv 1
\pmod{4}$. Let $Z(\mathbb{O}_r)$ be the center of $\mathbb{O}_r$. We
have $Z(\mathbb{O}_1)=O_F$, and $Z(\mathbb{O}_r)= \zz[\sqrt{p}]\neq
O_F$ for $r=8, 16$ when $p\equiv 1 \pmod{4}$. For the rest of this
subsection we write $A$ exclusively for the order $\zz[\sqrt{p}]$ when
$p\equiv 1 \pmod{4}$.  
Recall (Section~\ref{sec:5.2.1}) that 
$\varpi=[O_F^\times:A^\times]\in \{1, 3\}$, and $\varpi=1$ if $p\equiv
1 \pmod{8}$.

By the class number formula (\ref{eq:18}), 
\[h(\mathbb{O}_r)=\Mass(\mathbb{O}_r) +\Ell(\mathbb{O}_r) \qquad \text{ for
} r=1, 8, 16. \] The
mass part $\Mass(\mathbb{O}_r)$ has already been calculated in
Section~\ref{sec:special_case}. So we focus on the elliptic part
\[\Ell(\mathbb{O}_r)=\frac{1}{2} \sum_{B}(2-\delta(B))
h(B)(1-w(B)^{-1})\prod_\ell m(B_\ell, (\mathbb{O}_r)_\ell,
(\mathbb{O}_r)_\ell^\times), \] 
where $B$ runs through all the
(non-isomorphic) quadratic proper $Z(\mathbb{O}_r)$-orders with
\begin{equation}
  \label{eq:40}
 w(B)=[B^\times:Z(\mathbb{O}_r)^\times]>1,
\end{equation}
and $\delta(B)$ is given by (\ref{eq:19}), i.e. it is 1 if $B$ is
closed under the complex conjugation, and $0$ otherwise.

The detailed classification of all the orders $B$ will be given in the
subsequent sections.  We only summarize the results below. For this
purpose some more notations need to be introduced.

\begin{sect}\textbf{Notations of fields\footnote{If
    we need the 2-adic completion of a number field $K$, we will 
    write $K\otimes_\qq\qq_2$ instead of $K_2$ for the rest of the
    paper. Similarly for 3-adic completions.} and orders}.
  Let $K_j=\qq(\sqrt{p}, \sqrt{-j})$ with $j\in \{1,2,3\}$. 
One can show that
\begin{itemize}
\item for $p>5$, all quadratic $O_F$-orders $B$ with $[B^\times:
  O_F^\times]>1$ lie in $K_j$ for some $j\in \{1, 2,3\}$;
\item for $p\equiv 1 \pmod{4}$, all quadratic proper $A$-orders $B$
  with $[B^\times: A^\times]>1$ lie in either $K_1$ or $K_3$;
\end{itemize}
see \cite{xue-yang-yu:num_inv} for more details. 

We adopt the convention that $B_{j,k}$ is an order in $K_j$ with index
$k$ in $O_{K_j}$. The non-maximal suborders of $O_{K_j}$ that we will
consider are:
\begin{gather*}
  B_{1,2}:=\zz+\zz\sqrt{p}+\zz\sqrt{-1}+\zz(1+\sqrt{-1})(1+\sqrt{p})/2,
  \qquad B_{1,4}:=\zz[\sqrt{p},\sqrt{-1}],\\
  B_{3,4}:=\zz[\sqrt{p}, \zeta_6] \qquad \text{ if } p\equiv 1
  \pmod{4}; \\
  B_{3,2}:=A[\epsilon\zeta_6] \qquad \text{ if } p\equiv 5
  \pmod{8} \quad \text{and} \quad \varpi=3. 
\end{gather*}
Here $B_{3,2}$ is the suborder of $O_{K_3}$ generated by
$\epsilon\zeta_6$ over $A$, where $\epsilon\in O_F^\times$ is the
fundamental unit of $F$. With the exception of $B_{3,2}$, all the
other orders above are closed under the complex conjugation. 

Given a number field $K$, the class number of an arbitrary order
$B\subseteq O_K$ of conductor $\f$ can be computed by the following
formula \cite[Theorem I.12.12]{MR1697859}
  \begin{equation}
    \label{xeq:36}
    h(B)=\frac{h(O_K) [{(O_K/\f)}^\times:
      {(B/\f)}^\times]}{[O_K^\times :B^\times]}\,. 
  \end{equation}
\end{sect}


\begin{lem}\label{lem:bass-order}
  Assume that $p\equiv 1 \pmod{4}$. If $B\in \{B_{1,2}, B_{1,4},
  B_{3,4}, B_{3,2}\}$, then $B$ is a Bass order.
\end{lem}
\begin{proof}
  By a theorem of Borevich and Faddeev \cite{MR0205980}
  (cf. Curtis-Reiner \cite[Section 37, p.~789]{MR1038525}), $B$ is
  Bass if and only if the $B$-module $O_K/B$ is generated by one
  element. In particular, if $B$ is of prime index in $O_K$ then $B$
  is Bass. This shows that $B_{1,2}$ and $B_{3,2}$ are Bass
  orders. Since $O_{K_1}=\zz[\sqrt{-1},(1+\sqrt{p})/2]$, the quotient 
 $O_{K_1}/B_{1,4}$ is generated by  $(1+\sqrt{p})/2$ as a
  $B_{1,4}$-module. Hence $B_{1,4}$ is a Bass
  order. Since $2$ is inert in $\zz[\zeta_6]$, one has
  $O_{K_3}/B_{3,4}\simeq \ff_4$ as $\zz[\zeta_6]/(2)\simeq
  \ff_4$-modules. This proves that $B_{3,4}$ is also a Bass order.
\end{proof}

\begin{sect}\label{subsec:classno-max-order-general-p}
\textbf{Class number formula for $\mathbb{O}_1$ when $p>5$.}
  Since $\mathbb{O}_1$ is a
  maximal order and $D_{\infty_1, \infty_2}$ splits at all the finite
  places, we have $m(B_\ell,(\mathbb{O}_1)_\ell,
  (\mathbb{O}_1)_\ell^\times)=1$ for all $\ell$ (see \cite[p.~94]{vigneras}
  or Section~\ref{sec:15}). It follows that 
  \begin{equation}
    \label{eq:41}
\Ell(\mathbb{O}_1)=\frac{1}{2}\sum_{w(B)>1}h(B)(1-w(B)^{-1}),     
  \end{equation}
where  $w(B)=[B^\times: O_F^\times]$, and the summation is over all
isomorphism classes of quadratic $O_F$-orders $B$ with $w(B)>1$. 

When $p\equiv 1 \pmod{4}$
and $p>5$, the only orders with nonzero 
contributions to the elliptic
part $\Ell(\mathbb{O}_1)$ are $O_{K_1}$ and $O_{K_3}$, with
$w(O_{K_1})=2$ and $w(O_{K_3})=3$ respectively.  We have
\begin{equation}
  \label{eq:37}
  h(\mathbb{O}_1)= \frac{1}{2}h(F)\zeta_F(-1)+h(K_1)/4+h(K_3)/3,
\quad   \text{if}\  p\equiv 1 \pmod{4},\  p>5.
\end{equation}
When  $p \equiv 3 \pmod{4}$ and $p\geq 7$, we compute the following
numerical invariants
of all orders $B$ in some field $K_j$ with
$w(B)>1$: \\



\renewcommand{\arraystretch}{1.5}
\noindent\begin{tabular}{|>{$}c<{$}||*{5}{>{$}c<{$}|}}
\hline
  p\equiv 3 \pmod{4}
  & O_{K_1} & B_{1,2} & B_{1,4} & O_{K_2} & O_{K_3}\\
\hline
h(B) & h(K_1) & \left(2-\left(\frac{2}{p}\right)\right)h(K_1)
& \left(2-\left(\frac{2}{p}\right)\right)h(K_1) & h(K_2) & h(K_3)\\ 
w(B) & 4 & 4 & 2 & 2 & 3\\
\hline
\end{tabular}

\smallskip

Therefore, we have
\begin{equation}
  \label{eq:36}
  h(\mathbb{O}_1)= \frac{1}{2}h(F)\zeta_F(-1) +  \left(\frac{3}{8}+\frac{5}{8}\left(2-\left(\frac{2}{p}\right)\right)\right)h(K_1)+\frac{1}{4}h(K_2)+\frac{1}{3}h(K_3), 
\end{equation}
if  $p \equiv 3 \pmod{4}$ and $p\geq 7$. 
\end{sect}

\begin{sect}\textbf{Class number formula for $\mathbb{O}_8$ and
    $\mathbb{O}_{16}$ when $p\equiv 1\pmod{4}$.}
  Since $(\mathbb{O}_r)_\ell$ is maximal for all $\ell\neq 2$ and $r\in \{8,
  16\}$,  we have
  \begin{equation}\label{eq:42}
\Ell(\mathbb{O}_r)=\frac{1}{2}\sum_{w(B)>1}
  (2-\delta(B))h(B)(1-w(B)^{-1})m(B_2,(\mathbb{O}_r)_2, 
  (\mathbb{O}_r)_2^\times),
  \end{equation}
  where $w(B)=[B^\times: A^\times]$ and the summation is over all
  isomorphism classes of quadratic proper $A$-orders $B$ with
  $w(B)>1$. For simplicity, we will write
  $m_{2,r}(B):=m(B_2,(\mathbb{O}_r)_2, (\mathbb{O}_r)_2^\times)$ for
  $r=8,16$, where $(\oo_r)_2=\oo_r\otimes_\zz\zz_2$ and
  $B_2=B\otimes_\zz \zz_2$.  The numerical invariants of all proper
  $A$-orders $B$ with
  $w(B)>1$ are given by the following table: \\



\renewcommand{\arraystretch}{1.5}
\noindent\begin{tabular}{|>{$}c<{$}||*{4}{>{$}c<{$}|}}
\hline
  p\equiv 1 \pmod{4}
  &  B_{1,2} & B_{1,4} & B_{3,4} & B_{3,2}\\
\hline
 h(B) &  \frac{1}{\varpi}\left(2-\left(\frac{2}{p}\right)\right)h(K_1)
 & \frac{2}{\varpi}\left(2-\left(\frac{2}{p}\right)\right)h(K_1) &
 3h(K_3)/\varpi & h(K_3)\\
w(B) &  2 & 2 & 3 & 3\\
m_{2,8}(B)& 1 & 0 & 0 & 1 \\
m_{2,16}(B)& 0  &1  & 1& 0 \\
\delta(B)  & 1 & 1 & 1 & 0\\
\hline
\end{tabular} \\

 \smallskip

 Here $B_{3,2}$ is a proper $A$-order only if $p\equiv 5\pmod{8}$ and
 $\varpi=3$, in which case $\delta(B_{3,2})=0$. The numbers of
 conjugacy classes of 2-adic optimal embeddings $m_{2,r}(B)$ will be
 calculated in the next subsection.

For the explicit class number formulas of $\mathbb{O}_8$ and
$\mathbb{O}_{16}$, it is more convenient to separate into cases. If
$p\equiv 1 \pmod{8}$, then
\begin{align}
  h(\mathbb{O}_8)&=\frac{9}{2}\zeta_F(-1)h(F)+\frac{1}{4}h(K_1),
  \label{eq:38}\\ 
  h(\mathbb{O}_{16})&=3\zeta_F(-1)h(F)+\frac{1}{2}h(K_1)+h(K_3).
  \label{eq:39} 
\end{align}
 If $p\equiv
5 \pmod{8}$, then
\begin{align}
  h(\mathbb{O}_8)&=\frac{15}{2\varpi}\zeta_F(-1)h(F)+\frac{3}{4\varpi}
  h(K_1)+\frac{2\delta_{3,\varpi}}{\varpi} h(K_3), \label{eq:57} \\
  h(\mathbb{O}_{16})&=\frac{15}{\varpi}\zeta_F(-1)h(F)+ \frac{3}{2\varpi}
  h(K_1)+\frac{1}{\varpi}h(K_3),\label{eq:58}
\end{align}
where $\delta_{3,\varpi}$ is the Kronecker $\delta$-symbol.
\end{sect}

\begin{sect}\label{subsec:cal-local-opt-emb} \textbf{Numbers of
    conjugacy classes of 2-adic optimal embeddings.}  Assume that
  $p\equiv 1\pmod{4}$, and $B$ is an order in the list $\{B_{1,2}, B_{1,4}, B_{3,4}, B_{3,2}\}$. According to Lemma~\ref{lem:bass-order}, $B$ is a
  Bass order.  Recall that
  \[(\mathbb{O}_8)_2=\End_{A_2}(A_2\oplus O_{F_2}),\qquad
  (\mathbb{O}_{16})_2=\End_{A_2}(A_2^2) \] by the proof of
  Theorem~\ref{explicit_description}. It follows from
  Lemma~\ref{lem:local-optim-embedd-bass-order} that $m_{2,r}(B)\in
  \{0,1\}$ for $r=8,16$, and
\begin{align*}
 m_{2,8}(B)=1 &\iff B_2\simeq A_2\oplus O_{F_2}, \\
 m_{2,16}(B)=1 &\iff B_2\simeq A_2\oplus A_2.
\end{align*}

Since $A_2$ is a Bass order, $B_2$ is isomorphic to one of the
lattices given in (\ref{eq:63}). However, $B_2\not\simeq O_{F_2}\oplus
O_{F_2}$ as $B_2$ is a proper $A_2$-order.  Note that $O_FB=O_K$, where
the product is taken inside the fraction field $K$ of $B$. Hence
$B_2\otimes_{A_2}(A_2/2O_{F_2})\cong B_2/2(O_K\otimes_\zz\zz_2)\cong
B/2O_K$. By looking at the tensor product of $B_2$ with
$(A_2/2O_{F_2})$ for each $B$, we get the following isomorphisms of
$A_2$-modules
\[ (B_{1,2})_2\simeq (B_{3,2})_2 \simeq A_2\oplus O_{F_2}, \quad 
 (B_{1,4})_2\simeq (B_{3,4})_2\simeq A_2\oplus A_2. \]
As a result, we have
\[ m_{2,8}(B_{1,2})=1, \ m_{2,16}(B_{1,2})=0, \  m_{2,8}(B_{1,4})=0, \ 
m_{2,16}(B_{1,4})=1, \]
\[ m_{2,8}(B_{3,4})=0, \ m_{2,16}(B_{3,4})=1, \ m_{2,8}(B_{3,2})=1, \ 
m_{2,16}(B_{3,2})=0. \] 
\end{sect}

\begin{sect}\textbf{Special zeta-values.}
Let $\d_F$ be the discriminant of $F=\qq(\sqrt{p})$. 
By Siegel's formula \cite[Table 2, p.~70]{Zagier-1976-zeta},
\begin{equation}
  \label{eq:26}
  \zeta_F(-1)=\frac{1}{60}\sum_{\substack{b^2+4ac=\d_F\\ a,c>0}} a,
\end{equation}
where  $b\in \zz$ and $a,c\in \nn_{>0}$.
\end{sect}

It remains to calculate the class numbers of $\mathbb{O}_1$ when $p=2,3,
5$. This has already been done in \cite{Kirschmer-Voight} by computer. We 
list the results here for the
sake of completeness. 

\begin{sect}\label{subsec:class-numb-p-2}\textbf{Class number of $\mathbb{O}_1$ for $p=2$.} In this
  case $K_1=\qq(\sqrt{2}, \sqrt{-1})=\qq(\zeta_8)$. Besides $O_{K_1}$
  and $O_{K_3}$, we also need to consider the order $\zz[\sqrt{2},
  \sqrt{-1}]$, which is of index $2$ in $O_{K_1}$. The orders 
  with nonzero contributions to $\Ell(\mathbb{O}_1)$ are 

\renewcommand{\arraystretch}{1.3}
  \begin{tabular}{|>{$}c<{$}||*{3}{>{$}c<{$}|}}
\hline
   p=2 & \zz[\zeta_8] & \zz[\sqrt{2}, \sqrt{-1}] & \zz[\sqrt{2},
   \zeta_6]\\
\hline
   h(B) &  1 & 1 & 1  \\
  w(B)  &  4  & 2 & 3 \\
\hline
   \end{tabular}

\bigskip

\noindent Since $\zeta_{\qq(\sqrt{2})}(-1)=1/12$ by (\ref{eq:26}) and
$h(\qq(\sqrt{2}))=1$, 
\begin{equation}
  \label{eq:43}
  \begin{split}
  h(\mathbb{O}_1)&=\frac{1}{2}h(\qq(\sqrt{2}))\zeta_{\qq(\sqrt{2})}(-1)+\frac{1}{2}\left( \left(1-
    \frac{1}{4}\right)+\left(1- \frac{1}{2}\right)+\left(1-
    \frac{1}{3}\right) \right)\\
 & =\frac{1}{24}+\frac{23}{24}=1
  \qquad \text{ when } p =2.    
  \end{split}
 \end{equation}
\end{sect}


\begin{sect}\label{subsec:class-numb-p-3}
\textbf{Class number of $\mathbb{O}_1$ for $p=3$.}  In
  this case, we have $K_1=K_3=\qq(\zeta_{12})$. Besides the orders
  listed in the table of
  Section~\ref{subsec:classno-max-order-general-p}, we also need to
  consider the order $B_{1,3}:=\zz[\sqrt{3}, \zeta_6]$. The table
  becomes

  \begin{tabular}{|>{$}c<{$}||*{5}{>{$}c<{$}|}}
\hline
   p=3 & O_{K_1} & B_{1,2} & B_{1,4} & B_{1,3} & O_{K_2}\\
\hline
   h(B) &  1 & 1 & 1 & 1 & 2  \\
  w(B)  &  12  & 4 & 2 & 3 & 2 \\
\hline
   \end{tabular}

\bigskip

\noindent Hence \[\Ell(\mathbb{O}_1)=\frac{1}{2}\left( \left(1-
    \frac{1}{12}\right)+\left(1- \frac{1}{4}\right)+\left(1-
    \frac{1}{2}\right)+\left(1-
    \frac{1}{3}\right)+2\left(1-
    \frac{1}{2}\right) \right)=\frac{23}{12}. \] Using (\ref{eq:26})
again, $\zeta_{\qq(\sqrt{3})}(-1)=1/6$. Since
$h(\qq(\sqrt{3}))=1$, 
\begin{equation}
  \label{eq:44}
h(\mathbb{O}_1)=\frac{1}{2}h(\qq(\sqrt{3}))
\zeta_{\qq(\sqrt{3})}(-1)+\Ell(\mathbb{O}_1)
=\frac{1}{12}+\frac{23}{12}=2 
  \qquad \text{ when } p =3.   
\end{equation}
\end{sect}


\begin{sect}\label{subsec:class-numb-p-5}\textbf{Class number of
  $\mathbb{O}_1$ for $p=5$.}  
  In this case we also need to consider the field
  $\qq(\zeta_{10})$. The maximal order
  $\zz[\zeta_{10}]\subset \qq(\zeta_{10})$ is the only order whose 
  unit group is strictly
  larger than $O_F^\times$. The orders needed for the calculation of
  $\Ell(\mathbb{O}_1)$ are

  \begin{tabular}{|>{$}c<{$}||*{3}{>{$}c<{$}|}}
\hline
   p=5 & O_{K_1} & O_{K_3} & \zz[\zeta_{10}]\\
\hline
   h(B) &  1 & 1 & 1 \\
  w(B)  &  2  & 3 & 5 \\
\hline
   \end{tabular}

\bigskip

\noindent Since $\zeta_{\qq(\sqrt{5})}(-1)=1/30$ by (\ref{eq:26}) and
$h(\qq(\sqrt{5}))=1$,
\begin{equation}
  \label{eq:46}
  \begin{split}
      h(\mathbb{O}_1)&=\frac{1}{2}h(\qq(\sqrt{5}))
  \zeta_{\qq(\sqrt{5})}(-1)+\frac{1}{2}\left( \left(1-
    \frac{1}{2}\right)+\left(1- \frac{1}{3}\right)+\left(1-
    \frac{1}{5}\right)\right)\\
    &=\frac{1}{60}+\frac{59}{60}=1   \qquad \text{ when } p =5.   
  \end{split}
\end{equation}
\end{sect}

\begin{proof}[Proof of Theorem~\ref{1.2}]
  By definition, $H(p)=\abs{\Isog(X_\pi)}$, so it follows from
  Theorem~\ref{explicit_description} that 
\[
H(p)=\begin{cases}
  h(\oo_1)+h(\oo_8)+h(\oo_{16})&\qquad \text{ if } p\equiv 1 \pmod{4};\\
  h(\oo_1) &\qquad \text{ if } p\equiv 3 \pmod{4}\quad \text{or}
  \quad p=2.
\end{cases}
\]
The explicit formulae for $h(\oo_1)$ when $p=2$ and $p\equiv 3\pmod{4}$
have already been given above.

Suppose that $p=5$. We have $h(\oo_1)=1$ by
Section~\ref{subsec:class-numb-p-5}. The fundamental unit
$\epsilon=(1+\sqrt{5})/2\not\in \zz[\sqrt{5}]$, so $\varpi=3$. By
(\ref{eq:57}) and (\ref{eq:58}) respectively,
$h(\oo_8)=h(\oo_{16})=1$. Hence $H(p)=3$ if $p=5$.

Suppose that $p\equiv 1\pmod{8}$. Combining (\ref{eq:37}),
(\ref{eq:38}) and (\ref{eq:39}), we get
\[H(p)= h(\oo_1)+h(\oo_8)+h(\oo_{16})=8\zeta_F(-1)h(F)+
  h(K_1)+\frac{4}{3}h(K_3).  \]
Suppose that $p\equiv 5\pmod{8}$ and $p>5$. Note that
$2\delta_{3,\varpi}/\varpi+1/\varpi=1$ for $\varpi=1,3$. We obtain
\begin{equation*}
  \label{eq:59}
  \begin{split}
   H(p)&=\left(\frac{1}{2}+\frac{15}{2\varpi}+\frac{15}{\varpi}\right)\zeta_F(-1)h(F)+
  \left(\frac{1}{4}+\frac{3}{4\varpi}+\frac{3}{2\varpi}\right)h(K_1)+\frac{4}{3}h(K_3) \\ 
  &=\left(\frac{45+\varpi}{2\varpi}\right ) \zeta_F(-1)h(F)+
  \frac{9+\varpi}{4\varpi}h(K_1)+\frac{4}{3}h(K_3)
  \end{split}  
\end{equation*}
  by combining (\ref{eq:37}), (\ref{eq:57}) and (\ref{eq:58}). 
\end{proof}




\subsection{Asymptotic behavior} We keep the notations and
assumptions of Section~\ref{sec:12}.  In particular, $\{I_1,
\ldots, I_h\}$ is a complete set of representatives of the right ideal
classes $\Cl(\OO)$ of an order $\OO\subset D$ with center $Z(\OO)=A$.
The automorphism group $\Aut_\OO(I_i)$ of each $I_i$ as a right
$\OO$-module is $\OO_i^\times$, where $\OO_i=\OO_l(I_i)$. For an order
$\OO$ with a large number of ideal classes, it is generally expected
that $w_i=[\OO_i^\times:A^\times]=1$ for most  $1\leq i \leq
h$. Equivalently, we expect 
$\Mass(\OO)=\sum_{i=1}^h 1/w_i$ to be 
the dominant term in the class number formula 
$h(\OO)=\Mass(\OO)+\Ell(\OO)$. This is indeed the case
for the orders $\oo_r\subset D_{\infty_1, \infty_2}$ with $r=1,8,16$.
\begin{thm} Assume that that $p\equiv 1 \pmod{4}$ 
if $r=8,16$. For all $r\in
  \{1,8,16\}$, we have $ \lim_{p \to \infty} \Mass(\oo_r)/h(\oo_r)=1. $
\end{thm}
\begin{proof}
  It is enough to prove that $\lim_{p \to \infty}
  \Ell(\oo_r)/\Mass(\oo_r)=0$ for each $r$.  Recall that
  $\Mass(\oo_r)=c_r\zeta_F(-1)h(F)$, and $\Ell(\oo_r)=\sum_{j=1}^3
  d_{r,j}\,h(K_j)$ for some constants $c_r> 0$ and $d_{r,j}$ in each
  case. It reduces to prove that $\lim_{p \to \infty}
  h(K_j)/(\zeta_F(-1)h(F))=0$ for each $j\in \{1,2,3\}$. Let
  $\Bbbk_j=\qq(\sqrt{-pj})$, and $\d_{\Bbbk_j}$ be its discriminant.
  It follows from the work of Herglotz \cite{MR1544516}
  that $h(K_j)\leq h(F)h(\Bbbk_j)$ for $p\geq 5$ (see also
  Remark~\ref{1.4}).  We have $\lim_{p\to \infty} (\log
  h(\Bbbk_j))/(\log \sqrt{\abs{\d_{\Bbbk_j}}})=1$ by \cite[Theorem
  15.4, Chapter 12]{MR665428}. (See also \cite[Lemma
  4]{Horie-Horie-1990} for a similar result on the asymptotic behavior
  of relative class numbers of arbitrary CM-fields.)  On the other
  hand, $\zeta_F(-1)>(p-1)/240$ by (\ref{eq:26}). Hence \[0\leq
  \lim_{p\to \infty} \frac{h(K_j)}{h(F)\zeta_F(-1)}\leq \lim_{p\to
    \infty} \frac{h(\Bbbk_j)}{\zeta_F(-1)}=0,\] which shows that
  $\lim_{p \to \infty} h(K_j)/(h(F)\zeta_F(-1))=0$ for all $j\in
  \{1,2,3\}$.
\end{proof}


\section{Tables}
In this section, we list the class numbers $h(\oo_r)$ and related data
for $r=1, 8, 16$ (separated into 3 tables) and all primes $5<
p<200$. Here $F=\qq(\sqrt{p})$, and $K_j=\qq(\sqrt{p}, \sqrt{-j})$ for
$j=1,2,3$. Recall that $\oo_8$ and $\oo_{16}$ are defined only for
the primes $p\equiv 1\pmod{4}$. Moreover, for these $p$ the values of
$h(K_2)$ are not needed in the calculation and are left blank.  By
\cite[footnote to table 3, p.~424]{MR0195803}, out of the 303 primes $
p <2000$, $h(\qq(\sqrt{p}))=1$ for 264 of them.  So it is not
surprising that most $h(F)=1$ in Table~\ref{tab:table-O1}.

\begin{center}
\begin{longtable}{*{9}{|>{$}c<{$}}|}
  \caption{Class numbers of $\mathbb{O}_1$ for
    all primes $7\leq p<200$.}\label{tab:table-O1}\\
 \hline
 p & h(\mathbb{O}_1) & \Mass(\mathbb{O}_1) &\Ell(\mathbb{O}_1) & \zeta_F(-1)  & h(F) & h(K_1) & h(K_2) &
 h(K_3)\\ \hline
 \endhead
 \hline
 \multicolumn{9}{|c|}{{Continued on next page}} \\ \hline
\endfoot
\hline 
\endlastfoot
  7 &    3 &   1/3 &   8/3 &  2/3 &  1 &   1 &   4 &   2 \\
 11 &    4 &  7/12 & 41/12 &  7/6 &  1 &   1 &   2 &   2 \\
 13 &    1 &  1/12 & 11/12 &  1/6 &  1 &   1 &  &   2 \\
 17 &    1 &   1/6 &   5/6 &  1/3 &  1 &   2 &  &   1 \\
 19 &    6 & 19/12 & 53/12 & 19/6 &  1 &   1 &   6 &   2 \\
 23 &    7 &   5/3 &  16/3 & 10/3 &  1 &   3 &   4 &   4 \\
 29 &    2 &   1/4 &   7/4 &  1/2 &  1 &   3 &  &   3 \\
 31 &    9 &  10/3 &  17/3 & 20/3 &  1 &   3 &   8 &   2 \\
 37 &    2 &  5/12 & 19/12 &  5/6 &  1 &   1 &  &   4 \\
 41 &    2 &   2/3 &   4/3 &  4/3 &  1 &   4 &  &   1 \\
 43 &   12 &  21/4 &  27/4 & 21/2 &  1 &   1 &  10 &   6 \\
 47 &   13 &  14/3 &  25/3 & 28/3 &  1 &   5 &   8 &   4 \\
 53 &    3 &  7/12 & 29/12 &  7/6 &  1 &   3 &  &   5 \\
 59 &   16 & 85/12 & 107/12 & 85/6 &  1 &   3 &   6 &   2 \\
 61 &    3 & 11/12 & 25/12 & 11/6 &  1 &   3 &  &   4 \\
 67 &   18 &  41/4 &  31/4 & 41/2 &  1 &   1 &  14 &   6 \\
 71 &   19 &  29/3 &  28/3 & 58/3 &  1 &   7 &   4 &   4 \\
 73 &    3 &  11/6 &   7/6 & 11/3 &  1 &   2 &  &   2 \\
 79 &   69 &    42 &    27 &   28 &  3 &  15 &  24 &  18 \\
 83 &   22 &  43/4 &  45/4 & 43/2 &  1 &   3 &  10 &   6 \\
 89 &    4 &  13/6 &  11/6 & 13/3 &  1 &   6 &  &   1 \\
 97 &    4 &  17/6 &   7/6 & 17/3 &  1 &   2 &  &   2 \\
101 &    5 & 19/12 & 41/12 & 19/6 &  1 &   7 &  &   5 \\
103 &   31 &    19 &    12 &   38 &  1 &   5 &  20 &   6 \\
107 &   28 & 197/12 & 139/12 & 197/6 &  1 &   3 &   6 &  10 \\
109 &    5 &   9/4 &  11/4 &  9/2 &  1 &   3 &  &   6 \\
113 &    5 &     3 &     2 &    6 &  1 &   4 &  &   3 \\
127 &   39 &  80/3 &  37/3 & 160/3 &  1 &   5 &  16 &  10 \\
131 &   38 &  93/4 &  59/4 & 93/2 &  1 &   5 &   6 &   6 \\
137 &    6 &     4 &     2 &    8 &  1 &   4 &  &   3 \\
139 &   44 & 127/4 &  49/4 & 127/2 &  1 &   3 &  14 &   6 \\
149 &    7 & 35/12 & 49/12 & 35/6 &  1 &   7 &  &   7 \\
151 &   49 &    37 &    12 &   74 &  1 &   7 &  12 &   6 \\
157 &    7 & 43/12 & 41/12 & 43/6 &  1 &   3 &  &   8 \\
163 &   50 & 467/12 & 133/12 & 467/6 &  1 &   1 &  22 &  10 \\
167 &   47 &  91/3 &  50/3 & 182/3 &  1 &  11 &  12 &   8 \\
173 &    8 &  13/4 &  19/4 & 13/2 &  1 &   7 &  &   9 \\
179 &   54 & 157/4 &  59/4 & 157/2 &  1 &   5 &   6 &   6 \\
181 &    8 &  19/4 &  13/4 & 19/2 &  1 &   5 &  &   6 \\
191 &   61 & 130/3 &  53/3 & 260/3 &  1 &  13 &   8 &   8 \\
193 &   10 &  49/6 &  11/6 & 49/3 &  1 &   2 &  &   4 \\
197 &    9 & 49/12 & 59/12 & 49/6 &  1 &   5 &  &  11 \\
199 &   71 &    55 &    16 &  110 &  1 &   9 &  20 &   6 \\
    \hline
 \end{longtable}
\end{center}


\begin{table}[h]\caption{Class numbers of $\mathbb{O}_8$ for all
    primes $5< p<200$ and $p\equiv 1 \pmod{4}$.}\label{tab:table-O8}
  \begin{tabular}{rl}  
\hline
\multicolumn{1}{ |@{}r@{}|| }{ \begin{tabular}{>{$}c<{$}|>{$}c<{$}|>{$}c<{$}|>{$}c<{$}}
p & h(\mathbb{O}_8) & \Mass(\mathbb{O}_8)&\Ell(\mathbb{O}_8)\\
\hline
  13 &    2 &   5/12 & 19/12 \\
  17 &    2 &    3/2 &   1/2 \\
  29 &    4 &    5/4 &  11/4 \\
  37 &    7 &   25/4 &   3/4 \\
  41 &    7 &      6 &     1 \\
  53 &    7 &  35/12 & 49/12 \\
  61 &    8 &  55/12 & 41/12 \\
  73 &   17 &   33/2 &   1/2 \\
  89 &   21 &   39/2 &   3/2 \\
  97 &   26 &   51/2 &   1/2 \\
  \end{tabular}}&
\multicolumn{1}{ @{}l@{}| }
{  \begin{tabular}{>{$}c<{$}|>{$}c<{$}|>{$}c<{$}|>{$}c<{$}}
p & h(\mathbb{O}_8) & \Mass(\mathbb{O}_8)& \Ell(\mathbb{O}_8)\\
\hline
 101 &   29 &   95/4 &  21/4 \\
 109 &   16 &   45/4 &  19/4 \\
 113 &   28 &     27 &     1 \\
 137 &   37 &     36 &     1 \\
 149 &   21 & 175/12 & 77/12 \\
 157 &   24 & 215/12 & 73/12 \\
 173 &   24 &   65/4 &  31/4 \\
 181 &   29 &   95/4 &  21/4 \\
 193 &   74 &  147/2 &   1/2 \\
 197 &   65 &  245/4 &  15/4 \\
  \end{tabular}}\\
\hline
  \end{tabular}
\end{table}
\begin{table}[h]\caption{Class numbers of $\mathbb{O}_{16}$ for all
    primes $5< p<200$ and $p\equiv 1 \pmod{4}$.}\label{tab:table-O16}
  \begin{tabular}{rl}
\hline
\multicolumn{1}{ |@{}r@{}|| }{
  \begin{tabular}{>{$}c<{$}|>{$}c<{$}|>{$}c<{$}|>{$}c<{$}} 
  p & h(\mathbb{O}_{16}) & \Mass(\mathbb{O}_{16})&
  \Ell(\mathbb{O}_{16})\\ 
\hline
  13 &    2 &    5/6 &   7/6 \\
  17 &    3 &      1 &     2 \\
  29 &    5 &    5/2 &   5/2 \\
  37 &   18 &   25/2 &  11/2 \\
  41 &    7 &      4 &     3 \\
  53 &    9 &   35/6 &  19/6 \\
  61 &   12 &   55/6 &  17/6 \\
  73 &   14 &     11 &     3 \\
  89 &   17 &     13 &     4 \\
  97 &   20 &     17 &     3 \\
\end{tabular}}&
\multicolumn{1}{ @{}l@{}| }
{ \begin{tabular}{>{$}c<{$}|>{$}c<{$}|>{$}c<{$}|>{$}c<{$}}
p & h(\mathbb{O}_{16}) & \Mass(\mathbb{O}_{16})& \Ell(\mathbb{O}_{16})\\
\hline
 101 &   63 &   95/2 &  31/2 \\
 109 &   26 &   45/2 &   7/2 \\
 113 &   23 &     18 &     5 \\
 137 &   29 &     24 &     5 \\
 149 &   35 &  175/6 &  35/6 \\
 157 &   40 &  215/6 &  25/6 \\
 173 &   39 &   65/2 &  13/2 \\
 181 &   52 &   95/2 &   9/2 \\
 193 &   54 &     49 &     5 \\
 197 &  141 &  245/2 &  37/2 \\
\end{tabular}}\\
\hline
\end{tabular}
\end{table}

\section*{Acknowledgements}
The project grew from J.~Xue and CF Yu's participation 
in the Shimura curves
seminar organized by Yifan Yang at the the National Center for
Theoretical Science (NCTS). They also wish to thank NCTS for providing
Magma software that they use to compute the class numbers.  Discussions
with Markus Kirschmer, Meinhard Peters, Paul Ponomarev, 
John Voight, and Yifan Yang are very helpful
and greatly appreciated. J.~Xue is partially supported
by the grant NSC 102-2811-M-001-090. TC Yang and CF Yu 
are partially supported by the grants NSC
100-2628-M-001-006-MY4 and AS-98-CDA-M01.

\bibliographystyle{plain}
\bibliography{TeXBiB}
\end{document}